\documentclass{article}
\usepackage{here, amsmath, latexsym, amssymb, bm, ascmac, mathtools, multicol, tcolorbox, subfig, tikz, mathrsfs, siunitx, indentfirst}
\usepackage{amsthm}
\usepackage{tikz}
\usepackage[a4paper, total={6in, 9in}]{geometry}

\usepackage{titlesec}
\titleformat{\section}[block]{\bfseries\large}{\thesection.}{0.5em}{}
\titleformat{\subsection}[block]{\bfseries}{\thesubsection .}{0.5em}{\centering}

\theoremstyle{plain}
\newtheorem{thm}{Theorem}[section]
\newtheorem*{thm*}{Theorem}
\newtheorem{lem}[thm]{Lemma}

\newtheorem{prop}[thm]{Proposition}

\theoremstyle{definition}
\newtheorem{rem}{Remark}

\title{Critical Strichartz estimates for orthonormal systems}
\author{Yinghao Gong}

\begin{document}
	\maketitle
	\begin{abstract}
		Orthonormal Strichartz estimates for the free Schrödinger propagator take the form \begin{equation*}
			\left\Vert\sum_j\lambda_j\left|e^{it\Delta}f_j\right|^2\right\Vert_{L_t^\frac{q}{2}L_x^\frac{r}{2}(\mathbb{R}\times\mathbb{R}^d)}\lesssim\Vert\lambda\Vert_{\ell^\alpha(\mathbb{C})}
		\end{equation*}for arbitrary orthonormal systems $(f_j)_j$ in the homogeneous Sobolev space $\dot{H}^s(\mathbb{R}^d)$. In the admissible region, the optimal range of $\alpha$ has been established when $q\geq r/d'$. For $d\geq2$ and $2<q<r/d'$, it has been an open question as to determine whether the estimate holds when $\alpha=q/2$. We prove that the estimate holds in this critical case whenever $q=4$ and $r<\infty$. Our approach is robust and we illustrate this by extending the result to fractional Schrödinger propagators.
	\end{abstract}
	
	\section{Introduction.}
	Inequalities of the following form are called the \textit{(homogeneous) Strichartz estimates} for the Schrödinger operator $e^{it\Delta}$ on $\mathbb{R}^d$:\begin{equation}\label{classical Strichartz}
		\left\Vert e^{it\Delta}|\nabla|^{-s}f\right\Vert_{L_t^qL_x^r}\lesssim\Vert f\Vert_{L^2},
	\end{equation}where $L_t^qL_x^r=L_t^qL_x^r(\mathbb{R}\times\mathbb{R}^d)$, $L^2=L^2(\mathbb{R}^d)$. A necessary scaling condition for (\ref{classical Strichartz}) to hold is \begin{equation*}
		s=\frac{d}{2}-\frac{d}{r}-\frac{2}{q}.
	\end{equation*}We note that (\ref{classical Strichartz}) is true if $(q,r)$ lies in the closed region bounded by $2\leq q,r\leq\infty$ and $s\geq0$, except for the failure at $(q,r)=(\infty,\infty)$ for $d\geq1$, and the failure at $(q,r)=(2,\infty)$ for $d\geq2$, following the classic work of Strichartz \cite{Strichartz} and subsequent developments. (See, \cite{Ginibre-Velo,Montgomery-Smith,Keel-Tao,Z.Guo}.) We also note that (\ref{classical Strichartz}) fails if $(q,r)$ does not satisfy $2\leq q,r\leq\infty$ and $s\geq0$. A pair $(q,r)$ is called \textit{admissible} if (\ref{classical Strichartz}) holds for such exponents. 
	
	\begin{figure}[H]
		\centering
		\begin{tikzpicture}
			\draw [->, loosely dotted, thin] (0.05, 0) -- (2.5, 0) node [right] {$\frac{1}{r}$};
			\draw [->, loosely dotted, thin] (0, 0.05) -- (0, 2.5) node [above] {$\frac{1}{q}$};
			\fill [gray!40!white] (0, 0) -- (0, 1) -- (2, 0) -- cycle;
			\draw [solid, thick] (0.05, 0) -- (2, 0);
			\draw [solid, thick] (0, 0.05) -- (0, 1);
			\node [below left] (0, 0) {$0$};
			\draw [thick]  (0, 1) -- (2, 0);
			\draw [thick] (0,0) circle (0.05cm);
			\fill (2, 0) circle (0.05cm) node [below] {\tiny$(\frac{1}{2},0)$};
			\fill (0, 1) circle (0.05cm) node [left] {\tiny$(0,\frac{1}{4})$};
			\fill [white] (0, 0) circle (0.04cm);
		\end{tikzpicture}	
		\begin{tikzpicture}
			\draw [->, loosely dotted, thin] (0.05, 0) -- (2.5, 0) node [right] {$\frac{1}{r}$};
			\draw [->, loosely dotted, thin] (0, 0.05) -- (0, 2.5) node [above] {$\frac{1}{q}$};
			\fill [gray!40!white] (0, 0) -- (0, 2) -- (2, 0) -- cycle;
			\draw [solid, thick] (0.05, 0) -- (2, 0);
			\draw [solid, thick] (0, 0.05) -- (0, 1.95);
			\node [below left] (0, 0) {$0$};
			\draw [thick]  (0.03, 1.962) -- (2, 0);
			\draw [thick] (0,2) circle (0.05cm) node [left] {\tiny$(0,\frac{1}{2})$};
			\draw [thick] (0,0) circle (0.05cm);
			\fill (2, 0) circle (0.05cm) node [below] {\tiny$(\frac{1}{2},0)$};
			\fill [white] (0, 0) circle (0.04cm);
			\fill [white] (0, 2) circle (0.04cm);
		\end{tikzpicture}	
		\begin{tikzpicture}
			\draw [->, loosely dotted, thin] (0.05, 0) -- (2.5, 0) node [right] {$\frac{1}{r}$};
			\draw [->, loosely dotted, thin] (0, 0.05) -- (0, 2.5) node [above] {$\frac{1}{q}$};
			\fill [gray!40!white] (0, 0) -- (0, 2) -- (0.6, 2) -- (2, 0) -- cycle;
			\draw [solid, thick] (0.05, 0) -- (2, 0);
			\draw [solid, thick] (0, 0.05) -- (0, 1.95);
			\node [below left] (0, 0) {$0$}; 
			\draw [thick]  (0.6, 2) -- (2, 0);
			\draw [thick] (0,2) circle (0.05cm) node [left] {\tiny$(0,\frac{1}{2})$};
			\fill (0.6,2) circle (0.05cm) node [above right] {\footnotesize$(\frac{d-2}{2d},\frac{1}{2})$};
			\draw [thick] (0,0) circle (0.05cm);
			\draw [solid, thick] (0.05, 2) -- (0.6, 2);
			\fill (2, 0) circle (0.05cm) node [below] {\tiny$(\frac{1}{2},0)$};
			\fill [white] (0, 0) circle (0.04cm);
			\fill [white] (0, 2) circle (0.04cm);
		\end{tikzpicture}	
		\caption{The region where (\ref{classical Strichartz}) holds when $d=1$, $d=2$ and $d\geq3$, respectively.}
	\end{figure}
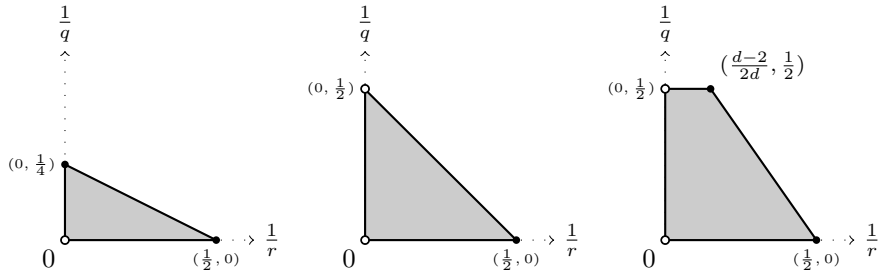
	
	Strichartz estimates serve as a fundamental tool and play a crucial role in the modern theory of nonlinear dispersive PDEs. By capturing the space-time decay properties of the free propagator, these estimates provide the essential framework required to establish the local and global well-posedness of various equations, including the nonlinear Schrödinger equation (NLS) and Hartree equations. A classical application of these estimates is in proving the well-posedness of critical NLS via a standard contraction mapping argument. (See, for example, \cite{C.Muscalu}, Section 11.) Furthermore, recent adaptations of these inequalities to orthonormal systems have opened up new pathways for analysing the dynamics of fermionic systems. In this paper, we are interested in the case of Schrödinger-type propagators. Inequalities of the following form are called the \textit{orthonormal Strichartz estimates} for the Schrödinger operator\begin{equation}\label{orthonormal Strichartz}
		\left\Vert\sum_j\lambda_j\left|e^{it\Delta}|\nabla|^{-s}f_j\right|^2\right\Vert_{L_t^\frac{q}{2}L_x^\frac{r}{2}}\lesssim\Vert\lambda\Vert_{\ell^\alpha},
	\end{equation}where $(f_j)\subset L^2$ is an arbitrary orthonormal system, and $\lambda=(\lambda_j)\in\ell^\alpha$ is an arbitrary sequence. We recall that $s=d/2-d/r-2/q$. We remark that (\ref{orthonormal Strichartz}) generally implies (\ref{classical Strichartz}), since for every non-zero datum $f\in L^2$, there always exists an orthonormal system $(f_j)\subset L^2$ with $f_1=f/\|f\|_{L^2}$, and taking $\lambda=(1,0,0,\ldots)$, we get\begin{equation*}
		\left\Vert\sum_j\lambda_j\left|e^{it\Delta}|\nabla|^{-s}f_j\right|^2\right\Vert_{L_t^\frac{q}{2}L_x^\frac{r}{2}}=\|f\|_{L^2}^{-2}\left\Vert e^{it\Delta}|\nabla|^{-s}f\right\Vert_{L_t^qL_x^r}^2.
	\end{equation*}
	
	As mentioned previously, inequality (\ref{orthonormal Strichartz}) is known to have numerous important applications to the nonlinear time-dependent Hartree equation. Indeed, in the context of many-body quantum mechanics, orthonormal systems are typically used to describe independent fermions in $\mathbb{R}^d$, whose dynamics are governed by Schrödinger equations. (See, for example, \cite{Chen,Frank1}.) Also, inequality (\ref{orthonormal Strichartz}) is known to imply a certain Besov space refinement of (\ref{classical Strichartz}). (See, \cite{Frank3}.) Motivated by the desire to understand dispersive effects in infinite quantum systems, Frank--Lewin--Lieb--Seiringer \cite{Frank1} were the first to study inequalities of the form (\ref{orthonormal Strichartz}). Later, Frank--Sabin \cite{Frank3} improved the result in \cite{Frank1} and we summarize the results from these papers in the following.\begin{thm}[\cite{Frank1,Frank3}]\label{FLLS}Let $d\geq1$. If \begin{equation*}
			\frac{d}{2}=\frac{d}{r}+\frac{2}{q};\qquad 2\leq r< 2+\frac{4}{d-1},
		\end{equation*}then\begin{equation}\label{first}
			\left\Vert\sum_j\lambda_j\left|e^{it\Delta}f_j\right|^2\right\Vert_{L_t^\frac{q}{2}L_x^\frac{r}{2}}\lesssim\Vert\lambda\Vert_{\ell^\frac{2r}{r+2}}
		\end{equation}holds for every orthonormal system $(f_j)\subset L^2$, and every sequence $\lambda=(\lambda_j)$. Moreover, the exponent $2r/(r+2)$ is sharp.
	\end{thm}
	
	\begin{rem}
		For an admissible pair $(q,r)$, we remark that applying (\ref{classical Strichartz}) trivially yields \begin{equation*}
			\left\Vert\sum_j\lambda_j\left|e^{it\Delta}|\nabla|^{-s}f_j\right|^2\right\Vert_{L_t^\frac{q}{2}L_x^\frac{r}{2}}\leq\sum_j\left|\lambda_j\right|\left\Vert e^{it\Delta}|\nabla|^{-s}f_j\right\Vert_{L_t^qL_x^r}^2\lesssim\Vert\lambda\Vert_{\ell^1},
		\end{equation*}even without using orthogonality. The point of Theorem \ref{FLLS} is that the orthogonality of the $f_j$ can give rise to better exponents on the right-hand side. Since $\|\lambda\|_{\ell^\alpha}\leq\|\lambda\|_{\ell^{\alpha'}}$ for $\alpha\geq\alpha'$, it makes sense to speak of the optimal exponent, and Theorem \ref{FLLS} is indeed sharp in this sense.	\end{rem}
	
	To ensure clarity and state results precisely, we introduce the following notation\begin{align*}
		O=(0&,0);\quad A=\left(\frac{d-1}{2d+2},\frac{d}{2d+2}\right);\quad B=\left(\frac{1}{2},0\right);\quad C=\left(0,\frac{1}{2}\right);\\&D=\left(\frac{d-2}{2d},\frac{1}{2}\right);\quad M=\left(0,\frac{1}{4}\right);\quad N=\left(\frac{d-1}{4d},\frac{1}{4}\right),
	\end{align*}and for two points $X_1$, $X_2$,\begin{align*}
		&(X_1,X_2)=\{(1-t)X_1+tX_2;t\in(0,1)\};\\
		&[X_1,X_2)=\{(1-t)X_1+tX_2;t\in[0,1)\};\\
		&[X_1,X_2]=\{(1-t)X_1+tX_2;t\in[0,1]\}.
	\end{align*}For more points $X_1,X_2,\ldots,X_n$, we write $X_1X_2\ldots X_n$ for the convex hull of these points. We use $(X_1X_2\ldots X_n)^o$ to denote the interior of $X_1X_2\ldots X_n$. 
	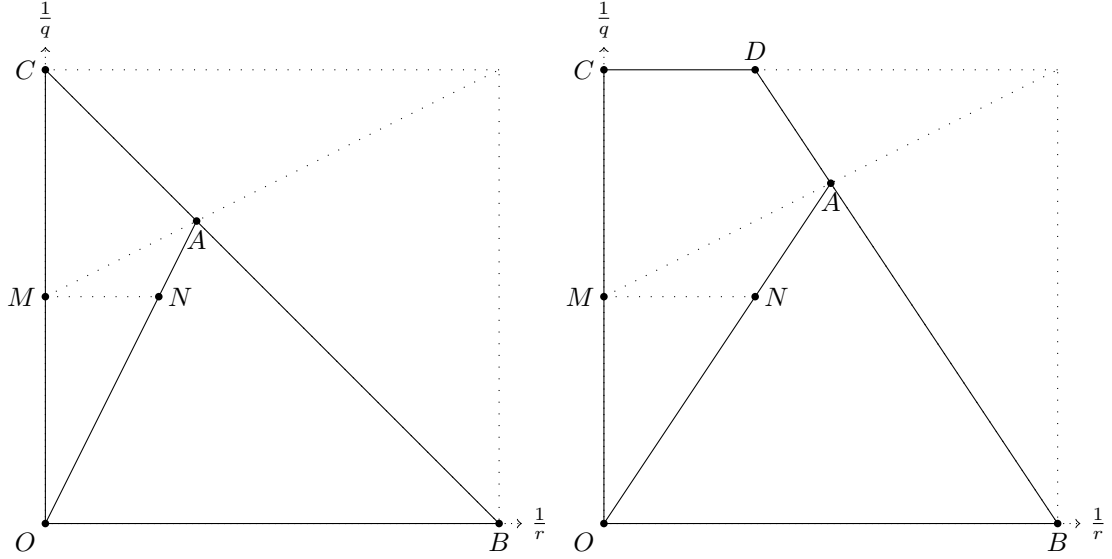
\begin{figure}[H]
		\centering
		\begin{tikzpicture}
			\draw [->, dotted, thin] (0, 0) -- (6.3, 0) node [right] {$\frac{1}{r}$};
			\draw [->, dotted, thin] (0, 0) -- (0, 6.3) node [above] {$\frac{1}{q}$};
			\draw [loosely dotted] (6, 6) -- (0, 6) -- (6, 0) --cycle;
			\draw (0, 0) -- (0, 6) -- (6, 0) -- cycle;
			\draw [loosely dotted] (0, 3) -- (6, 6);
			\draw (0, 0) -- (2, 4);
			\draw [loosely dotted] (0, 3) -- (1.5, 3);
			\fill (0, 0) circle (0.05cm) node [below left] {$O$};
			\fill (2, 4) circle (0.05cm) node [below] {$A$};
			\fill (6, 0) circle (0.05cm) node [below] {$B$};
			\fill (0, 6) circle (0.05cm) node [left] {$C$};
			\fill (0, 3) circle (0.05cm) node [left] {$M$};
			\fill (1.5, 3) circle (0.05cm) node [right] {$N$};
		\end{tikzpicture}\begin{tikzpicture}
			\draw [->, dotted, thin] (0, 0) -- (6.3, 0) node [right] {$\frac{1}{r}$};
			\draw [->, dotted, thin] (0, 0) -- (0, 6.3) node [above] {$\frac{1}{q}$};
			\draw [loosely dotted] (6, 6) -- (2, 6) -- (6, 0) --cycle;
			\draw (0, 0) -- (0, 6) -- (2, 6) -- (6, 0) -- cycle;
			\draw [loosely dotted] (0, 3) -- (6, 6);
			\draw (0, 0) -- (3, 4.5);
			\draw [loosely dotted] (0, 3) -- (2, 3);
			\fill (0, 0) circle (0.05cm) node [below left] {$O$};
			\fill (3, 4.5) circle (0.05cm) node [below] {$A$};
			\fill (6, 0) circle (0.05cm) node [below] {$B$};
			\fill (0, 6) circle (0.05cm) node [left] {$C$};
			\fill (2, 6) circle (0.05cm) node [above] {$D$};
			\fill (0, 3) circle (0.05cm) node [left] {$M$};
			\fill (2, 3) circle (0.05cm) node [right] {$N$};
		\end{tikzpicture}
		\caption{Notation in the case $d=2$ (left), and $d\geq3$ (right).}
	\end{figure}
	Let $\alpha_0$ be given by \begin{equation*}
		\frac{d}{\alpha_0}=\frac{2}{q}+\frac{2d}{r}.
	\end{equation*}Define $\alpha^\ast(q,r)=\min\{\alpha_0,q/2\}$. We note that on $OAB$, $\alpha^\ast(q,r)=\alpha_0$; on $OCDA$, $\alpha^\ast(q,r)=q/2$; on the segment $[O,A]$, $\alpha^\ast(q,r)=\alpha_0=q/2$. 
	
	We write $\mathcal{O}_2(q,r,\alpha)$ to mean (\ref{orthonormal Strichartz}) holds for arbitrary orthonormal systems $(f_j)\subset L^2$ and arbitrary sequences $\lambda=(\lambda_j)\in\ell^\alpha$. Here we review some established results for the orthonormal Strichartz framework.
	
	\begin{thm}[\cite{Frank1,Frank2,Frank3,Bez,Bez2}]\label{review}
		The following are true.
		\begin{itemize}
			\item[(I).] Let $d=1$. Then, for $(1/r,1/q)\in(OAB)^o\cup[B,A)$, $\mathcal{O}_2(q,r,\alpha)$ holds if and only if $1\leq\alpha\leq\alpha^\ast(q,r)$; for $(1/r,1/q)\in(O,B)\cup(O,A)$, $\mathcal{O}_2(q,r,\alpha)$ holds if and only if $1\leq\alpha<\alpha^\ast(q,r)$.
			\item[(II).] Let $d\geq2$. Then, for $(1/r,1/q)\in(OAB)^o\cup[B,A)$, $\mathcal{O}_2(q,r,\alpha)$ holds if and only if $1\leq\alpha\leq\alpha^\ast(q,r)$; for $(1/r,1/q)\in(O,B)\cup(O,A]$, $\mathcal{O}_2(q,r,\alpha)$ holds if and only if $1\leq\alpha<\alpha^\ast(q,r)$.
			\item[(III).] Let $d\geq2$. Then, for $(1/r,1/q)\in(OCDA)^o\cup(D,A)\cup(O,C)$, $\mathcal{O}_2(q,r,\alpha)$ holds if $1\leq\alpha<\alpha^\ast(q,r)$, and fails if $\alpha>\alpha^\ast(q,r)$.
			\item[(IV).]Let $d\geq3$. Then, for $(1/r,1/q)\in(C,D]$, $\mathcal{O}_2(q,r,\alpha)$ holds if and only if $\alpha=1$.
		\end{itemize}
	\end{thm}
	
	From Theorem \ref{review}, for $d\geq2$, what remains open is to establish whether $\mathcal{O}_2(q,r,q/2)$ holds or not for $(1/r,1/q)\in(OCDA)^o\cup(D,A)\cup(O,C)$. In this paper, we provide progress on this by proving the following.

	\begin{thm}\label{main1}
		Let $d\geq2$. Then, $\mathcal{O}_2(4,r,2)$ holds if $4d/(d-1)<r<\infty$.
	\end{thm}
	
	\begin{rem}
		When $d\geq2$, the condition $q=4$, $4d/(d-1)<r<\infty$ corresponds to the segment $(M,N)$.
	\end{rem}
	\begin{rem}
		The authors of \cite{Feng} stated that when $d\geq2$,  $\mathcal{O}_2(q,r,\alpha^\ast(q,r))$ holds on the entire region of $(OCDA)^o$, by combining the argument in \cite{Bez} with the following real interpolation identity\begin{equation}\label{error}
			(L^{q_0,\infty}L^{p_0},L^{q_1,\infty}L^{p_1})_{\theta,q}=L^qL^{p,q},
		\end{equation}where $1/p=(1-\theta)/p_0+\theta/p_1$, $1/q=(1-\theta)/q_0+\theta/q_1$, whenever $p_0,p_1,q_0,q_1,p,q\in[1,\infty]$, $\theta\in(0,1)$ satisfy $p_0\neq p_1$, $q_0\neq q_1$. (See, \cite{Feng}, (2.2).) However, to the best of our knowledge, the identity (\ref{error}) is generally false. The simplist counterexample can be given by the case $p_0=q_0$, $p_1=q_1$, if we compute norms of the specific function $f(t,x) = 1_{(0,1)}(t)1_{E_t}(x)$, where $E_t$ is a rectangle with dimensions $1 \times 1/t$, so that its Lebesgue measure is $|E_t| = 1/t$. It is easy to show that $f\in L^{q_0,\infty}L^{p_0} $ and $f\in L^{q_1,\infty}L^{p_1} $. By the fundamental properties of real interpolation, the intersection of two spaces is continuously embedded into their interpolation space for any $\theta \in (0,1)$. Therefore, $f \in (L^{q_0,\infty}L^{p_0},L^{q_1,\infty}L^{p_1})_{\theta,q}$. On the other hand, our assumption $p_0=q_0$ and $p_1=q_1$ forces $q=p$, which means $\|f\|_{L_t^qL_x^{p,q}}=\|f\|_{L_t^pL_x^{p}}=\infty$, and thus we show that (\ref{error}) is not true in general.
	\end{rem}
	We will provide two completely different proofs of Theorem \ref{main1}. The first is based on estimates established by Chen--Hong--Pavlović \cite{Chen} and we provide the details in Section 2. The second proof is inspired by ideas in Keel--Tao \cite{Keel-Tao}. The main advantage of the second proof is that it is robust and we use this robustness to extend Theorem \ref{main1} to the class of fractional Schrödinger propagators.
	
	\vspace{2ex}
	
	The \textit{fractional Schrödinger propagator} $e^{it|\nabla|^{\mu}}$ on $\mathbb{R}^d$ is defined by the following \begin{equation*}
		e^{it|\nabla|^{\mu}} f(x)\simeq_d\int_{\mathbb{R}^d}e^{i(x\cdot\xi+t|\xi|^\mu)}\widehat{f}(\xi)d\xi.
	\end{equation*}We note that $\mu=2$ corresponds to $e^{-it\Delta}$. The fractional Schrödinger equation, formulated by Laskin \cite{Laskin1} to generalize standard quantum mechanics to Lévy processes, has attracted a lot of attention in recent years. 
	(See, \cite{Laskin1,Laskin2}.) 
	
	\vspace{2ex}
	
	Strichartz estimates for fractional Schrödinger propagators are given by the following form. \begin{equation}\label{classical frac Strichartz}
		\left\Vert e^{it|\nabla|^{\mu}}|\nabla|^{-s_\mu}f\right\Vert_{L_t^qL_x^r}\lesssim\Vert f\Vert_{L^2}.
	\end{equation}We remark that for $\mu\in\mathbb{R}\setminus\{0,1\}$, the inequality (\ref{classical frac Strichartz}) holds if and only if $(q,r)$ is an admissible pair. (See, \cite{Z.Guo}.) Inequalities of the following form are called the \textit{orthonormal Strichartz estimates} for the fractional Schrödinger propagator  \begin{equation}\label{XXX}
		\left\Vert\sum_j\lambda_j\left|e^{it|\nabla|^{\mu}}|\nabla|^{-s_\mu}f_j\right|^2\right\Vert_{L_t^\frac{q}{2}L_x^\frac{r}{2}}\lesssim\Vert\lambda\Vert_{\ell^\alpha},
	\end{equation}where $(f_j)\subset L^2$ is an arbitrary orthonormal system, $\lambda=(\lambda_j)$ is an arbitrary sequence, and the scaling exponent $s_\mu$ is given by\begin{equation*}
		s_\mu=\frac{d}{2}-\frac{d}{r}-\frac{\mu}{q}.
	\end{equation*}We write $\mathcal{O}_\mu(q,r,\alpha)$ to mean (\ref{XXX}) holds for arbitrary orthonormal systems $(f_j)\subset L^2$ and arbitrary sequences $\lambda=(\lambda_j)\in\ell^\alpha$. For $\mu>0$, $\mu\neq1$, and except some critical cases, it is well understood when $\mathcal{O}_\mu(q,r,\alpha)$ holds. (See, \cite{Bez2, Bez3}.) However, $\mathcal{O}_\mu(q,r,q/2)$ remains open on $(OCDA)^o$ and in this paper we extend Theorem \ref{main1} to all $\mu>0$, $\mu\neq1$.
	
	\begin{thm}\label{main2}
		Let $d\geq2$, $\mu>0$, $\mu\neq1$. Then, $\mathcal{O}_\mu(4,r,2)$ holds if $4d/(d-1)<r<\infty$.
	\end{thm}

	\begin{rem}
		According to a duality principle (Lemma 3, \cite{Frank3}) proved by Frank--Sabin, the inequality $\mathcal{O}_\mu(4,r,2)$ is equivalent with \begin{equation}\label{dualise1}
			\left\Vert We^{it|\nabla|^{\mu}}|\nabla|^{-2s_\mu}e^{-it|\nabla|^{\mu}}\overline{W}\right\Vert_{\mathfrak{S}^2}\lesssim\Vert W\Vert_{L_t^4 L_x^{\widetilde{r}}}^2,\text{ for all } W \in L_t^4 L_x^{\widetilde{r}},
		\end{equation}where $\widetilde{r}$ satisfies $1/\widetilde{r}=1/2-1/r$. Here, $\mathfrak{S}^2=\mathfrak{S}^2(L^2)$ is a Hilbert--Schmidt space, which is the most accessible Schatten space from a computational viewpoint, and this explains why we restrict to the line segment $(M,N)$. Notice that (\ref{dualise1}) has a natural bilinear structure and this will be important when we use ideas of Keel--Tao \cite{Keel-Tao} to prove Theorem \ref{main1} and Theorem \ref{main2}.
	\end{rem}

	From now on, we let $d\geq2$ and let $(q,r)$ be admissible. In this paper, we do not distinguish in the notation between functions $W$ and operators of multiplication by $W$. We recall that $A\lesssim B$ means $A\leq CB$ where $C$ is a positive constant, more generally, $A\lesssim_{q,r}B$ means $A\leq C_{q,r}B$ where $C_{q,r}$ is a positive constant depending only on $q$ and $r$. Also, $A\sim B$ means $A\lesssim B$ and $B\lesssim A$, $A\simeq B$ means $A=CB$ where $C$ is a positive constant. For an operator $T:L^2\longrightarrow L^2$, we define the \textit{Schatten $p$-norm} $\|T\|_{\mathfrak{S}^p}=\|T\|_{\mathfrak{S}^p(L^2)}=\|\left(\lambda_n(T)\right)\|_{\ell^p}$, where $\lambda_n(T)$ are the square roots of the eigenvalues of $T^\ast T$. The \textit{Schatten space} for $L^2$ is defined as $\mathfrak{S}^p=\mathfrak{S}^p(L^2)=\{T:L^2\longrightarrow L^2;\|T\|_{\mathfrak{S}^p}<\infty\}$. We recall that the Schatten $2$-norm is a Hilbert--Schmidt norm, which means $\|T\|_{\mathfrak{S}^2}=\|\ker_T\|_{L^2}$, where $\ker_T$ is the integral kernel function of $T$. For further details about Schatten spaces, we refer the reader to \cite{Simon}.

	\section{A proof of Theorem \ref{main1} following \cite{Chen}.}
	In this section, we introduce a proof of Theorem \ref{main1} following the original method in \cite{Chen}. For a fixed orthonormal system $(f_j)\subset L^2$, for $\lambda_1,\ldots,\lambda_N\in\mathbb{C}$, define \begin{equation*}
		\gamma_0f=\sum_{j=1}^{N}\lambda_j\left<f,f_j\right>_{L^2}f_j,
	\end{equation*}where $\left<\cdot,\cdot\right>_{L^2}$ is the standard inner product on $L^2$. We also define the Schrödinger flow $\gamma(t)=e^{it\Delta}\gamma_0e^{-it\Delta}$. Indeed, $\gamma(t)$ satisfies the quantum Liouville equation (the von Neumann-Schrödinger equation) $i\dot{\gamma}(t)=[-\Delta,\gamma(t)]$, and we refer the reader to \cite{Chen,Frank3} for further physical details about this equation and the Hartree equation. 
	
	\vspace{2ex}
	
	Consider $\gamma(t)$. The orthogonality of $(f_j)$ yields\begin{equation*}
		\gamma(t)f(x)=e^{it\Delta}\gamma_0e^{-it\Delta}f(x)=\int_{\mathbb{R}^d}f(y)\sum_{j=1}^N\lambda_j\overline{e^{it\Delta}f_j(y)}e^{it\Delta}f_j(x)dy.
	\end{equation*}Thus, \begin{equation*}
		\ker_{\gamma(t)}(x,x)=\sum_{j=1}^N\lambda_j\left|e^{it\Delta}f_j(x)\right|^2,
	\end{equation*}and also, we derive\begin{equation*}
		\|\gamma_0\|_{\mathfrak{S}^\alpha}=\left(\sum_{j=1}^N|\lambda_j|^\alpha\right)^\frac{1}{\alpha}.
	\end{equation*}
	
	\vspace{2ex}
	
	By introducing the density function $\rho_{\gamma(t)}(x)=\ker_{\gamma(t)}(x,x)$ and following this structure, we reformulate $\mathcal{O}_2(q,r,\alpha)$ as \begin{equation}\label{gamma1}
		\left\|\rho_{\gamma(t)}\right\|_{L_t^\frac{q}{2}L_x^\frac{r}{2}}\lesssim\||\nabla|^{s}\gamma_0|\nabla|^{s}\|_{\mathfrak{S}^\alpha},\text{ for all } \gamma_0\in\mathfrak{S}^\alpha.
	\end{equation}We will see that to show $\mathcal{O}_2(4,r,2)$, it suffices to prove the following estimate.
	\begin{thm}\label{CHPs}
		Let $d\geq2$. Suppose $\omega,\omega_1$ satisfy\begin{equation*}
			\frac{d-1}{4}<\omega<\frac{d-1}{2};\quad 2\omega_1-4\omega+d-1=0.
		\end{equation*}Then, the inequality \begin{equation}\label{gamma2}
			\||\nabla|^{\omega_1+\frac{1}{2}}\rho_{\gamma(t)}\|_{L_t^2L_x^2}\lesssim\||\nabla|^\omega\gamma_0|\nabla|^\omega\|_{\mathfrak{S}^2}
		\end{equation}is true for every $\gamma_0\in\mathfrak{S}^2$. 
	\end{thm}
	
	\begin{rem}
		Let us see how Theorem \ref{CHPs} implies Theorem \ref{main1}. Let $s=d/2-d/r-1/2$ and define $\omega_1$ such that\begin{equation*}
			\frac{2}{r}=\frac{1}{2}-\frac{2\omega_1+1}{2d}.
		\end{equation*}By applying the Sobolev embedding theorem, we obtain\begin{equation*}
			\|\rho_{\gamma(t)}\|_{L_t^2L_x^\frac{r}{2}}\lesssim\||\nabla|^{\omega_1+\frac{1}{2}}\rho_{\gamma(t)}\|_{L_t^2L_x^2}.
		\end{equation*}On the other hand, Theorem \ref{CHPs} implies \begin{equation*}
			\||\nabla|^{\omega_1+\frac{1}{2}}\rho_{\gamma(t)}\|_{L_t^2L_x^2}\lesssim\||\nabla|^s\gamma_0|\nabla|^s\|_{\mathfrak{S}^2}
		\end{equation*}for all $4d/(d-1)<r<\infty$. Thus, by (\ref{gamma1}), we see that $\mathcal{O}_2(4,r,2)$ holds. In other words, Theorem \ref{main1} follows from Theorem \ref{CHPs}.
	\end{rem}
	
	\begin{rem}
		The following inequality (see, \cite[(3.17)]{Chen}) has been shown to be true and plays a key role during the global well-posedness study by Chen--Hong--Pavlović in \cite{Chen}\begin{equation*}
			\left\|\left<\nabla\right>^{\omega_1}|\nabla|^{\frac{1}{2}}\rho_{\gamma(t)}\right\|_{L_t^2L_x^2}\lesssim\|\left<\nabla\right>^\omega\gamma_0\left<\nabla\right>^\omega\|_{\mathfrak{S}^2},
		\end{equation*}where \begin{equation*}
			d\geq2;\qquad\frac{d-1}{4}<\omega<\frac{d-1}{2};\qquad 2\omega_1-4\omega+d-1=0.
		\end{equation*}The following proof completely mirrors the original proof of (3.17) in \cite{Chen}. We include details in order to clarify that the inhomogeneous derivatives can be replaced by homogeneous derivatives. Also, by giving the details, we can explain why the proof seems difficult to extend to other propagators, such as the fractional Schrödinger propagators.
	\end{rem}
	
	\begin{proof}[Proof of Theorem \ref{CHPs}.]
		Denote $F(t,x)=\rho_{\gamma(t)}(x)$, and $K(x,x')$ as the kernel of $\gamma_0$. Note that $F(t,x)$ is formed as \begin{equation*}
			F(t,x)\simeq\int_{\mathbb{R}^{2d}}e^{-it(|\xi_1|^2-|\xi_2|^2)}\widehat{K}(\xi_1,\xi_2)e^{ix\cdot(\xi_1+\xi_2)}d\xi_1d\xi_2.
		\end{equation*}Then, by changing variables $\xi=\xi_1+\xi_2$, we obtain\begin{equation*}
			\widetilde{F}(\tau,\xi)\simeq\int_{\mathbb{R}^d}\widehat{K}(\xi_1,\xi-\xi_1)\delta(\tau+|\xi_1|^2-|\xi-\xi_1|^2)d\xi_1.
		\end{equation*}Here, $\widetilde{F}$ is the Fourier transform of $F$ on $\mathbb{R}\times\mathbb{R}^d$, and $\delta$ is the delta distribution supported at zero. On the other hand, we use duality to rewrite (\ref{gamma2}) as  \begin{equation*}
			\left|\langle|\nabla|^{\omega_1+\frac{1}{2}}F(t,x),V(t,x)\rangle_{L_t^2L_x^2}\right|\lesssim\||\nabla_x|^\omega|\nabla_{x'}|^\omega K\|_{L_x^2L_{x'}^2}\|V\|_{L_t^2L_x^2},\text{ for all } V\in L_t^2L_x^2.
		\end{equation*}Plancherel's theorem thus implies that (\ref{gamma2}) is equivalent with\begin{equation}\label{CHP2}
			\left|\langle|\xi|^{\omega_1+\frac{1}{2}}\widetilde{F}(\tau,\xi),\widetilde{V}(\tau,\xi)\rangle_{L_\tau^2L_\xi^2}\right|\lesssim\||\xi_1|^\omega|\xi_2|^\omega \widehat{K}\|_{L_{\xi_1}^2L_{\xi_2}^2}\|V\|_{L_t^2L_x^2}, \text{ for all }V\in L_t^2L_x^2.
		\end{equation}Since \begin{equation*}
			\left|\langle|\xi|^{\omega_1+\frac{1}{2}}\widetilde{F}(\tau,\xi),\widetilde{V}(\tau,\xi)\rangle_{L_\tau^2L_\xi^2}\right|\simeq\left|\int_{\mathbb{R}^{2d}}|\xi|^{\omega_1+\frac{1}{2}}\widehat{K}(\xi_1,\xi-\xi_1)\overline{\widetilde{V}(-|\xi_1|^2+|\xi-\xi_1|^2,\xi)}d\xi d\xi_1\right|,
		\end{equation*}we see that to show (\ref{CHP2}), it suffices to show\begin{equation}\label{CHP3}
			\sqrt{I}:=\left\|\frac{|\xi|^{\omega_1+\frac{1}{2}}\widetilde{V}(|\xi|^2-2\xi\cdot\xi_1,\xi)}{|\xi_1|^\omega|\xi-\xi_1|^\omega}\right\|_{L_{\xi_1}^2L_{\xi}^2}\lesssim\|V\|_{L_t^2L_x^2},
		\end{equation}where we used the fact that\begin{equation}\label{fact}
			-|\xi_1|^2+|\xi-\xi_1|^2=|\xi|^2-2\xi\cdot\xi_1.
		\end{equation}To prove (\ref{CHP3}), first observe that there exists some $\eta=\eta(\xi)=(\eta_1,\ldots,\eta_d)\in\mathbb{R}^d$ such that $\xi_1=\eta_1\mathbf{e}_1+\cdots+\eta_d\mathbf{e}_d$, where $(\mathbf{e}_j)$ is an orthonormal basis of $\mathbb{R}^d$ which satisfies $\xi\bot\text{span}(\mathbf{e}_2,\ldots,\mathbf{e}_d)$. We have $\xi=|\xi|\mathbf{e}_1$, consequently,\begin{equation*}
			\xi\cdot\xi_1=\eta_1|\xi|;\qquad|\xi-\xi_1|=|(\eta_1-|\xi|,\eta')|,
		\end{equation*}where $\eta'=(\eta_2,\cdots,\eta_d)\in\mathbb{R}^{d-1}$. Regarding $I$, we then write \begin{equation*}
			I=\int_{\mathbb{R}^d}\left(\int_{\mathbb{R}^{d-1}}\int_{\mathbb{R}}\frac{|\xi|^{2\omega_1+1}|\widetilde{V}(|\xi|^2-2\eta_1|\xi|,\xi)|^2}{|\eta|^{2\omega}|(\eta_1-|\xi|,\eta')|^{2\omega}}d\eta_1d\eta'\right)d\xi.
		\end{equation*}For a fixed $\xi$, we set $\tau=|\xi|^2-2\eta_1^\ast|\xi|$, which means, \begin{equation*}
			\eta_1^\ast=\eta_1^\ast(\tau,\xi)=-\frac{\tau-|\xi|^2}{2|\xi|}.
		\end{equation*}Then, by changing variables, applying Hölder's inequality and Plancherel's theorem, we deduce\begin{align*}
			I=&\int_{\mathbb{R}^d}\int_{\mathbb{R}}\left(\int_{\mathbb{R}^{d-1}}\frac{|\xi|^{2\omega_1}}{2|(\eta_1^\ast,\eta')|^{2\omega}|(\eta_1^\ast-|\xi|,\eta')|^{2\omega}}d\eta'\right)|\widetilde{V}(|\xi|^2-2\eta_1^\ast|\xi|,\xi)|^2d\tau d\xi\\\leq&\|V\|_{L_t^2L_x^2}^2\sup_{\tau,\xi}\int_{\mathbb{R}^{d-1}}\frac{|\xi|^{2\omega_1}}{2|(\eta_1^\ast,\eta')|^{2\omega}|(\eta_1^\ast-|\xi|,\eta')|^{2\omega}}d\eta'.
		\end{align*}To show (\ref{CHP3}), it suffices to show that \begin{equation}\label{CHP4}
			\sup_{\tau,\xi}\int_{\mathbb{R}^{d-1}}\frac{|\xi|^{2\omega_1}}{2|(\eta_1^\ast,\eta')|^{2\omega}|(\eta_1^\ast-|\xi|,\eta')|^{2\omega}}d\eta'<\infty
		\end{equation}holds. First, we restrict $2|\eta_1^\ast|\leq|\xi|$ on the region $|\eta'|\leq2|\xi|$. In this case, $2|\eta_1^\ast-|\xi\|\geq|\xi|$, so that for every $\omega<(d-1)/2$, we have\begin{align*}
			\int_{|\eta'|\leq2|\xi|}\frac{|\xi|^{2\omega_1}}{2|(\eta_1^\ast,\eta')|^{2\omega}|(\eta_1^\ast-|\xi|,\eta')|^{2\omega}}d\eta'&\leq\int_{|\eta'|\leq2|\xi|}\frac{|\xi|^{2\omega_1}}{2|\eta'|^{2\omega}|\eta_1^\ast-|\xi||^{2\omega}}d\eta'\\&\lesssim\int_{|\eta'|\leq2|\xi|}\frac{|\xi|^{2\omega_1-2\omega}}{|\eta'|^{2\omega}}d\eta'\\&\lesssim|\xi|^{2\omega_1-4\omega+d-1}\\&=1.
		\end{align*}On the other hand, if we restrict $2|\eta_1^\ast|\geq|\xi|$ on the same region $|\eta'|\leq2|\xi|$, then for every $\omega<(d-1)/2$, still we have \begin{align*}
			\int_{|\eta'|\leq2|\xi|}\frac{|\xi|^{2\omega_1}}{2|(\eta_1^\ast,\eta')|^{2\omega}|(\eta_1^\ast-|\xi|,\eta')|^{2\omega}}d\eta'&\leq\int_{|\eta'|\leq2|\xi|}\frac{|\xi|^{2\omega_1}}{2|\eta_1^\ast|^{2\omega}|\eta'|^{2\omega}}d\eta'\\&\lesssim\int_{|\eta'|\leq2|\xi|}\frac{|\xi|^{2\omega_1-2\omega}}{|\eta'|^{2\omega}}d\eta'\\&\lesssim|\xi|^{2\omega_1-4\omega+d-1}\\&=1.
		\end{align*}
		On the remaining region $|\eta'|\geq2|\xi|$, if $\omega>(d-1)/4$, then we obtain \begin{equation*}
			\int_{|\eta'|\geq2|\xi|}\frac{|\xi|^{2\omega_1}}{2|(\eta_1^\ast,\eta')|^{2\omega}|(\eta_1^\ast-|\xi|,\eta')|^{2\omega}}d\eta'\leq\int_{|\eta'|\geq2|\xi|}\frac{|\xi|^{2\omega_1}}{|\eta'|^{4\omega}}d\eta'\leq|\xi|^{2\omega_1-4\omega+d-1}=1.
		\end{equation*}In conclusion, the inequality (\ref{CHP4}) holds if\begin{equation*}
			\frac{d-1}{4}<\omega<\frac{d-1}{2};\quad 2\omega_1-4\omega+d-1=0,
		\end{equation*}which finishes this proof.
	\end{proof}
	
	\begin{rem}\label{R6}
		The above proof is highly effective due to the identity (\ref{fact}), which arises since we are considering $e^{it\Delta}$. For general fractional Schrödinger propagators, it seems difficult to establish Theorem \ref{main2} by using such an approach.
	\end{rem}
	
	\section{A proof of Theorem \ref{main1} using bilinear interpolation.}
	In this section, we establish an original approach for $\mathcal{O}_2(4,r,2)$ on the segment $(M,N)$. We start the argument from a frequency localised setting. Denote $U_s=e^{it\Delta}|\nabla|^{-s}$, where $s=d/2-d/r-1/2$. Fix a non-negative $\phi_1\in C_c^\infty(\mathbb{R}^d)$ with $\phi_1(0)=1$, and define $\phi_R(\xi)=\phi_1(R^{-1}\xi)$. Define a Fourier multiplier $P_R$ such that $\widehat{P_Rf}(\xi)=\phi_R(\xi)\widehat{f}(\xi)$, and denote $U_{s,R}=e^{it\Delta}|\nabla|^{-s}P_R$. Consider the following frequency localised version inequality of $\mathcal{O}_2(4,r,2)$\begin{equation}\label{fl1}
		\left\Vert\sum_j\lambda_j\left|U_{s,R}f_j\right|^2\right\Vert_{L_t^2L_x^\frac{r}{2}}\lesssim\Vert\lambda\Vert_{\ell^2},
	\end{equation}where $(f_j)\subset L^2$ is an orthonormal system, and $\lambda=(\lambda_j)\subset\mathbb{C}$ is a sequence. According to \cite[Lemma 3]{Frank3}, (\ref{fl1}) is equivalent with\begin{equation}\label{fl2}
		\left\Vert WU_{s,R}U_{s,R}^\ast\overline{W}\right\Vert_{\mathfrak{S}^2}\lesssim\Vert W\Vert_{L_t^4 L_x^{\widetilde{r}}}^2,
	\end{equation}where $U_{s,R}^\ast$ is the adjoint of $U_{s,R}$, and the implicit constant is independent of $R$. Note that we may deduce $O_2(4,r,2)$ by applying Fatou's lemma to (\ref{fl2}). Thus, to show Theorem \ref{main1}, it suffices to proof (\ref{fl2}). 
	
	\vspace{2ex}
	
	Notice that we have \begin{equation*}
		\left\Vert WU_{s,R}U_{s,R}^\ast\overline{W}\right\Vert_{\mathfrak{S}^2}=R^{-(2s+2)}\left\Vert W_RU_{s,1}U_{s,1}^\ast\overline{W_R}\right\Vert_{\mathfrak{S}^2},
	\end{equation*}where $W_R(t,x)=W(R^{-2}t,R^{-1}x)$. Thus, by proving the following lemma and using normal rescaling argument, we obtain (\ref{fl2}).  \begin{lem}\label{31}
		Let $d\geq2$. Suppose $2<\widetilde{r}<4d/(d+1)$. Then, the inequality\begin{equation}\label{3A}
			\left\Vert W_1U_{s,1}U_{s,1}^\ast W_2\right\Vert_{\mathfrak{S}^2}\lesssim\Vert W_1\Vert_{L_t^4 L_x^{\widetilde{r}}}\Vert W_2\Vert_{L_t^4 L_x^{\widetilde{r}}}
		\end{equation}holds for every $W_1,W_2 \in L_t^4 L_x^{\widetilde{r}}$.
	\end{lem}
	In Subsection 3.1, we derive some estimates for the integral kernel of $U_{s,1}U_{s,1}^\ast$, by using the oscillatory integral theory. In Subsection 3.2, we show how these estimates establish Lemma \ref{31}, by applying the interpolation method in Keel--Tao \cite{Keel-Tao} to the bilinear structure $W_1U_{s,1}U_{s,1}^\ast W_2$.
	
	\subsection{Oscillatory integral estimates.}\label{s1}
	Firstly, we derive the form of $\ker_{U_{s,1}U_{s,1}^\ast}$. Since \begin{equation*}
		\left<\widehat{f},\widehat{U_{s,1}^\ast F} \right>_{L^2}\simeq\left<f,U_{s,1}^\ast F\right>_{L^2}\simeq\left<U_{s,1}f,F\right>_{L^2},
	\end{equation*}we deduce\begin{equation*}
		\widehat{U_{s,1}^\ast F}(\xi)\simeq\int_{\mathbb{R}^{d+1}}e^{-i(x\cdot\xi-t|\xi|^2)}\phi_1(\xi)|\xi|^{-s}F(t,x)dxdt=\phi_1(\xi)|\xi|^{-s}\widehat{F}(-|\xi|^2,\xi).
	\end{equation*}Thus, \begin{equation*}
		U_{s,1}U_{s,1}^\ast F(t,x)\simeq\int_{\mathbb{R}^{d+1}}F(t',y)\int_{\mathbb{R}^d}e^{i\left((x-y)\cdot\xi-(t-t')|\xi|^2\right)}\phi_1^2(\xi)|\xi|^{-2s}d\xi dydt',
	\end{equation*}we obtain\begin{equation*}
		\ker_{U_{s,1}U_{s,1}^\ast}(t,t',x,y)=\int_{\mathbb{R}^d}e^{i\left((x-y)\cdot\xi-(t-t')|\xi|^2\right)}\phi_1^2(\xi)|\xi|^{-2s}d\xi,
	\end{equation*}and thus, \begin{align*}
		&\left\Vert W_1U_{s,1}U_{s,1}^\ast W_2\right\Vert_{\mathfrak{S}^2}^2\\=&\left\Vert \ker_{W_1U_{s,1}U_{s,1}^\ast W_2}\right\Vert_{L^2(\mathbb{R}^{2d+2})}^2\\=&\int_{\mathbb{R}^{2d+2}}|W_1(t,x)|^2|W_2(t',y)|^2\left|\int_{\mathbb{R}^d}e^{i\left((x-y)\cdot\xi-(t-t')|\xi|^2\right)}\phi_1^2(\xi)|\xi|^{-2s}d\xi\right|^2 dxdydtdt'.
	\end{align*}To compensate for the singularity caused by $\phi_1^2(\xi)|\xi|^{-2s}$ at the origin, we apply the following dyadic decomposition. Let $\psi\in C_c^\infty$ satisfy \begin{equation*}
		\text{supp}\psi=\left\{\xi;\frac{1}{4}\leq|\xi|\leq\frac{1}{2}\right\};\qquad\phi_1^2(\xi)=\sum_{j=1}^{\infty}\psi(2^j\xi).
	\end{equation*}Define the following bilinear transformation \begin{equation*}
		T(F,G)=\int_{\mathbb{R}^{2d+2}}F(t,x)G(t',y)\sum_{j_0,j_1} \left|I_{j_0}(t-t',x-y)I_{j_1}(t-t',x-y)\right|dxdydtdt',
	\end{equation*}where \begin{equation*}
		I_j(t,x)=2^{-j(d-2s)}\int_{\mathbb{R}^d}e^{i\left(2^{-j}x\cdot\xi-2^{-2j}t|\xi|^2\right)}\psi(\xi)|\xi|^{-2s}d\xi.
	\end{equation*}By applying the triangle inequality, we see that to prove Lemma \ref{31}, it suffices to show that \begin{equation}\label{3B}
		|T(F,G)|\lesssim\| F\|_{L_t^2L_x^p}\| G\|_{L_t^2L_x^p}
	\end{equation}holds for all $1<p<2d/(d+1)$, and for all functions $F,G\in L_t^2 L_x^p$. 
	
	\vspace{2ex}
	
	Regarding the phase function $\Phi(\xi)=\Phi_{t,x}(\xi)=2^{-j}x\cdot\xi-2^{-2j}t|\xi|^2$ of $I_j$, the triangle inequality and our assumption $1/4\leq|\xi|\leq1/2$ imply $|\nabla\Phi(\xi)|\geq2^{-j}|x|-2^{-2j}|t|$, and $|\nabla\Phi(\xi)|\geq2^{-1}\cdot2^{-2j}|t|-2^{-j}|x|$. Thus,\begin{equation}\label{3C}
		|\nabla\Phi(\xi)|\gtrsim\begin{cases}
			2^{-j}|x|,& 2|t|\leq2^j|x|;\\2^{-2j}|t|,& 2^j|x|\leq2^{-2}|t|.
		\end{cases}
	\end{equation}In order to apply the oscillatory integral theory, we introduce the following notation for collections of $j$ with fixed $x,y,t,t'$. \begin{itemize}
		\item $\mathbf{A}_1=\{j;1/4|t-t'|\leq2^j|x-y|\leq2|t-t'|\}$;
		\item $\mathbf{A}_2=\{j;2|t-t'|\leq2^j|x-y|\}$;
		\item $\mathbf{A}_3=\{j;2^j|x-y|\leq1/4|t-t'|\}$.
	\end{itemize}We estimate $I_j$ by the following statement.
	
	\begin{lem}\label{32}
		Let $d\geq2$. Suppose $(d+1)/2d<1/p<1$. Then, \begin{equation}\label{3D}
			\sum_{j_0\in\mathbf{A}_m}\sum_{j_1\in\mathbf{A}_n} \left|I_{j_0}(t-t',x-y)I_{j_1}(t-t',x-y)\right|\lesssim|x-y|^{-2\delta}|t-t'|^{-(d-2s-\delta)}
		\end{equation}holds for all of the following cases\begin{itemize}
			\item[(B1).] $m=n=1$, and $\delta\in\{2s-d,2s\}$;
			\item[(B2).]  $m,n\in\{2,3\}$, and $\delta\in\{0,d-2s\}$;
			\item[(B3).] $(m,n)\in\{(1,2),(1,3),(2,1),(3,1)\}$, and $\delta\in\{0,d/2\}$.
		\end{itemize}
	\end{lem}
	
	\vspace{2ex}
	
	To prove Lemma \ref{32}, the following proposition is useful.
	\begin{prop}\label{dispersive estimates}The following inequalities are true.
		\begin{itemize}
			\item[(i)] For every $t,x$,\begin{equation*}
				|I_j(t,x)|\lesssim2^{-j(d-2s)};
			\end{equation*}
			\item[(ii)] For $1/4|t|\leq2^j|x|\leq2|t|$,\begin{equation*}
				|I_j(t,x)|\lesssim2^{-j(d-2s)}\frac{1}{(2^{-j}|x|)^{d/2}};
			\end{equation*}
			\item[(iii)] For $2|t|\leq2^j|x|$,\begin{equation*}
				|I_j(t,x)|\lesssim_N2^{-j(d-2s)}\frac{1}{(2^{-j}|x|)^{N}}\lesssim_N2^{-j(d-2s)}\frac{1}{(2^{-2j}|t|)^{N}}
			\end{equation*}holds for every $N>0$;
			\item[(iv)] For $2^j|x|\leq1/4|t|$,\begin{equation*}
				|I_j(t,x)|\lesssim_N2^{-j(d-2s)}\frac{1}{(2^{-2j}|t|)^{N}}\lesssim_N2^{-j(d-2s)}\frac{1}{(2^{-j}|x|)^{N}}
			\end{equation*}holds for every $N>0$.
		\end{itemize}
	\end{prop}
	
	Since $\psi$ has compact support and vanishes near the origin, trivially we obtain (i). Applying (\ref{3C}) and Proposition 2.1 in \cite[Chapter 8]{M.Stein 1} yields (iii) and (iv). Regarding the singularity region $1/4|t|\leq2^j|x|\leq2|t|$, one has $\det\{\nabla^2(\Phi/|t|)\}\sim_d2^{-2jd}$, thus, we obtain (ii) by applying Proposition 2.5 in \cite[Chapter 8]{M.Stein 1}. 
	
	\begin{proof}[Proof of Lemma \ref{32}]
		Fix $x,y,t,t'$. First, we consider the case $m=n=1$. Note that the condition\begin{equation*}
			\frac{1}{4}|t-t'|\leq2^j|x-y|\leq2|t-t'|
		\end{equation*}implies $\# \mathbf{A}_1\sim 1$. Also, for $j\in\mathbf{A}_1$, \begin{equation*}
			2^j\sim \frac{|t-t'|}{|x-y|}.
		\end{equation*}Thus, using (i) yields \begin{equation*}
			\sum_{j_0\in\mathbf{A}_1}\sum_{j_1\in\mathbf{A}_1} \left|I_{j_0}(t-t',x-y)I_{j_1}(t-t',x-y)\right|\lesssim\left(\frac{|x-y|}{|t-t'|}\right)^{2d-4s},\end{equation*}and on the other hand, using (ii) yields\begin{align*}
			\sum_{j_0\in\mathbf{A}_1}\sum_{j_1\in\mathbf{A}_1} \left|I_{j_0}(t-t',x-y)I_{j_1}(t-t',x-y)\right|\lesssim&\sum_{j_0\in\mathbf{A}_1}\sum_{j_1\in\mathbf{A}_1}2^{-(j_0+j_1)(\frac{d}{2}-2s)}|x-y|^{-d}\\\sim&\left(\frac{|x-y|}{|t-t'|}\right)^{d-4s}|x-y|^{-d}\\=&|x-y|^{-4s}|t-t'|^{-(d-4s)}.
		\end{align*}This completes the proof for case $m=n=1$.
		
		\vspace{2ex}
		
		Consider the other cases. Here we remark that the following are given by applying (i), (iii) and (iv): \begin{align}
			\sum_{j\in\{2|t|\leq2^j|x|\}\bigcup\{2^j|x|\leq1/4|t|\}}|I_j(t,x)|&\lesssim\frac{1}{|x|^{d-2s}};\label{7}\\
			\sum_{j\in\{2|t|\leq2^j|x|\}\bigcup\{2^j|x|\leq1/4|t|\}}|I_j(t,x)|&\lesssim\frac{1}{|t|^{\frac{d-2s}{2}}}.\label{8}
		\end{align}
		In fact, (\ref{7}) is true, since \begin{align*}
			\sum_{\footnotesize\begin{matrix}
					j\in\{2|t|\leq2^j|x|\}\bigcup\{2^j|x|\leq1/4|t|\}\\|x|\geq2^j
			\end{matrix}}|I_j(t,x)|&\lesssim\frac{1}{|x|^{N}}\sum_{|x|\geq2^j}2^{j(N-d+2s)}\lesssim\frac{1}{|x|^{d-2s}};\\\sum_{\footnotesize\begin{matrix}
					j\in\{2|t|\leq2^j|x|\}\bigcup\{2^j|x|\leq1/4|t|\}\\|x|\leq2^j
			\end{matrix}}|I_j(t,x)|&\lesssim\sum_{|x|\leq2^j}2^{-j(d-2s)}\sim\frac{1}{|x|^{d-2s}}
		\end{align*}are true, if we take $N>d-2s$, and apply (iii) and (iv). Also, (\ref{8}) can be given similarly.
		
		\vspace{2ex}
		
		By using inequalities (\ref{7}) and (\ref{8}), we may directly complete the proof of the remaining cases. For instance, by applying (\ref{7}) and (ii), we deduce \begin{align*}
			\sum_{j_0\in\mathbf{A}_1}\sum_{j_1\in\mathbf{A}_2\cup\mathbf{A}_3} \left|I_{j_0}(t-t',x-y)I_{j_1}(t-t',x-y)\right|&\lesssim2^{-j_0(\frac{d}{2}-2s)}|x-y|^{-\frac{d}{2}}|x-y|^{-(d-2s)}\\&\sim|x-y|^{-d}|t-t'|^{-(\frac{d}{2}-2s)}.
		\end{align*}Computations for other cases are similar and simple, so we omit them here.
	\end{proof}
	
	\vspace{2ex}
	
	Let us go back to consider (\ref{3B}). Define two characteristic functions $\nu_0(x,y)=1_{|x-y|^2\leq|t-t'|}(x,y)$, and $\nu_1(x,y)=1_{|x-y|^2\geq|t-t'|}(x,y)$. Define the following operator \begin{equation*}
		T_\sigma^{m,n}(F,G)=\int_{\mathbb{R}^{2d+2}}F(t,x)G(t',y)\nu_\sigma(x,y)\sum_{j_0\in\mathbf{A}_m}\sum_{j_1\in\mathbf{A}_n} \left|I_{j_0}(t-t',x-y)I_{j_1}(t-t',x-y)\right|dxdydtdt',
	\end{equation*}where $m,n\in\{1,2,3\}$, $\sigma\in\{0,1\}$. To prove (\ref{3B}), it suffices to show the following inequality \begin{equation}\label{3E}
		\left|T_\sigma^{m,n}(F,G)\right|\lesssim\| F\|_{L_t^2L_x^p}\| G\|_{L_t^2L_x^p}
	\end{equation}holds for every $m,n\in\{1,2,3\}$, $\sigma\in\{0,1\}$.
	
	\begin{rem}
		We cannot derive (\ref{3E}) directly by applying Lemma \ref{32}. For instance, to estimate $T_1^{1,1}$, we can take $\delta=2s$ in (B1) and use Young's convolution inequality to yield \begin{align*}
			\left|T_1^{1,1}(F,G)\right|\lesssim&\int_{\mathbb{R}^{2}}\int_{|x-y|^2\geq|t-t'|}|F(t,x)||G(t',y)||x-y|^{-4s}|t-t'|^{-(d-4s)}dxdydtdt'\\=&\int_{\mathbb{R}^{2}}\int_{|x-y|^2\geq1}\left|F(t,|t-t'|^\frac{1}{2}x)\right|\left|G(t',|t-t'|^\frac{1}{2}y)\right||x-y|^{-4s}|t-t'|^{2s}dxdydtdt'\\\leq&\int_{\mathbb{R}^{2}}\Vert F(t,\cdot)\Vert_{L_x^p}\Vert G(t',\cdot)\Vert_{L_x^p}|t-t'|^{-1}dtdt'.
		\end{align*}We see that this inequality holds if and only if $1< p<2d/(d+1)$. In fact, applying Young's convolution inequality and appropriate estimates in Lemma \ref{32} will give \begin{equation}\label{fail}
			\left|T_\sigma^{m,n}(F,G)\right|\lesssim\int_{\mathbb{R}^{2}}\Vert F(t,\cdot)\Vert_{L_x^p}\Vert G(t',\cdot)\Vert_{L_x^p}|t-t'|^{-1}dtdt'
		\end{equation}for all $1\leq p<2d/(d+1)$, for all cases $m,n\in\{1,2,3\}$, $\sigma\in\{0,1\}$. However, we cannot obtain (\ref{3E}) from (\ref{fail}), since the Hardy--Littlewood--Sobolev inequality fails to apply with respect to $|t-t'|^{-1}$. A similar obstacle arises in the proof of the classical Strichartz estimates (\ref{classical Strichartz}) at the endpoint $(q,r)=(2,2d/(d-2))$, $d\geq3$. This problem was elegantly settled by Keel--Tao through a highly original bilinear interpolation argument. (See, \cite{Keel-Tao}, Section 6.) In Subsection 3.2, we show that following a similar strategy to \cite{Keel-Tao} allows us to bypass this obstacle and prove Theorem \ref{main1}.
	\end{rem}
	
	\subsection{The bilinear interpolation method.}\label{s2}
	Recall that our goal is to prove (\ref{3E}). Take $1<p^\ast<2d/(d+1)$ and fix it. Let $s^\ast$ be given by $2s^\ast=d/p^\ast-1$. Choose $\kappa\in C_c^\infty(\mathbb{R})$ with \begin{equation*}
		\text{supp}\kappa=[1,2];\quad\sum_{k\in\mathbb{Z}}\kappa(2^{-k}t)=1.
	\end{equation*}The existence of $\kappa$ is easily justified. Denote $\kappa_k(t)=\kappa(2^{-k}t)$, and define a bilinear operator
	\begin{equation*}
		\mathcal{T}_{k,\sigma,\delta,p^\ast}(F,G)=\int_{\mathbb{R}^{2d+2}}\kappa_k(t-t')F(t,x)G(t',y)\nu_\sigma(x,y)|x-y|^{-2\delta}|t-t'|^{-(d-2s^\ast-\delta)}dxdydtdt'.
	\end{equation*}Note that for every $m,n\in\{1,2,3\}$, if $\delta$ is chosen as we determined in Lemma \ref{32}, then (\ref{3D}) implies \begin{equation*}
		|T_\sigma^{m,n}(F,G)|\lesssim\sum_{k\in\mathbb{Z}}|\mathcal{T}_{k,\sigma,\delta,p^\ast}(F,G)|
	\end{equation*}for every fixed $p^\ast$ and every $\sigma\in\{0,1\}$. Thus, to prove (\ref{3E}), it suffices to show that the following bilinear mapping 
	\begin{equation*}
		\mathtt{T}_{\sigma,\delta,p^\ast}(F,G)=\left(\mathcal{T}_{k,\sigma,\delta,p^\ast}(F,G)\right)_{k\in\mathbb{Z}}
	\end{equation*}is $L_t^2L_x^p\times L_t^2L_x^p\longrightarrow\ell^1$ bounded for every $(\sigma,\delta)\in\{(0,0),(0,2s^\ast-d),(1,2s^\ast),(1,d-2s^\ast),(1,d/2)\}$. For this, the following estimates are key.
	\begin{lem}\label{33}
		Let $d\geq2$. Suppose $1<p^\ast<2d/(d+1)$. Let $a,b$ be such that $1/a+1/b>1$ and satisfy\begin{equation*}
			\max\left\{\frac{d+1}{d}-\frac{1}{p^\ast},\frac{1}{p^\ast}-\frac{1}{d}\right\}<\frac{1}{a}<1\qquad;\qquad\max\left\{\frac{d+1}{d}-\frac{1}{p^\ast},\frac{1}{p^\ast}-\frac{1}{d}\right\}<\frac{1}{b}<1.
		\end{equation*}Define \begin{equation*}
			\beta(a,b)=d+1-\frac{d}{2}\left(\frac{1}{a}+\frac{1}{b}\right).
		\end{equation*}Then, for every pair $(\sigma,\delta)\in\{(0,0),(0,2s^\ast-d),(1,2s^\ast),(1,d-2s^\ast),(1,d/2)\}$, the inequality \begin{equation}\label{3F}
			|\mathcal{T}_{k,\sigma,\delta,p^\ast}(F,G)|\lesssim2^{-k\gamma(a,b)}\Vert F\Vert_{L_t^2 L_x^a}\Vert G\Vert_{L_t^2 L_x^b}
		\end{equation}is true for every $F\in L_t^2 L_x^a$, $G\in L_t^2 L_x^b$, where $\gamma(a,b)=\beta(p^\ast,p^\ast)-\beta(a,b)$.
	\end{lem}
	
	\begin{rem}
		The above lemma is still insufficient to provide $\ell^1$ estimates directly for $\mathtt{T}_{\sigma,\delta,p^\ast}$, since $\gamma(p^\ast,p^\ast)=0$. We employ the following bilinear interpolation proposition (as a black box) to gain the necessary room to obtain this $\ell^1$ bound.
		\begin{prop}[\cite{Bergh}, Section 3.13, 5-b]\label{Bergh-Löfström}
			For given Banach spaces $A_0$, $A_1$, $B_0$, $B_1$, $C_0$, $C_1$, if a bilinear transformation	
			\begin{equation}\label{3G}
				L:
				\begin{cases}
					A_0\times B_0\longrightarrow C_0 \\ A_0\times B_1\longrightarrow C_1 \\ A_1\times B_0\longrightarrow C_1
				\end{cases}
			\end{equation}
			is bounded, then
			\begin{equation}\label{3H}
				L:(A_0,A_1)_{\theta_0,p_0}\times (B_0,B_1)_{\theta_1,p_1}\longrightarrow(C_0,C_1)_{\theta,1}
			\end{equation}
			is also bounded whenever $p_0,p_1\in[1,\infty]\text{ satisfy }1/p_0+1/p_1\geq 1\text{, and }\theta_0,\theta_1,\theta\in (0,1)$ satisfy $\theta=\theta_0+\theta_1$.
		\end{prop}
	\end{rem}
	We provide a proof of Lemma \ref{33} later. Firstly, let us see how to deduce (\ref{3E}) by using Lemma \ref{33} and Proposition \ref{Bergh-Löfström}. Take a sufficiently small $\varepsilon>0$, let $a_0,a_1$ satisfy $1/a_0=1/p^\ast-\varepsilon$, $1/a_1=1/p^\ast+2\varepsilon$. Set
	\begin{equation*}
		A_0=B_0= L_t^2 L_x^{a_0},\quad A_1=B_1= L_t^2 L_x^{a_1},\quad C_0=\ell_{\gamma(a_0,a_0)}^\infty,\quad C_1=\ell_{\gamma(a_0,a_1)}^\infty.
	\end{equation*}
	Thus, if Lemma \ref{33} is proved to be true, then $\mathtt{T}_{\sigma,\delta,p^\ast}$ satisfies (\ref{3G}). Denote $Q_1=(1/a_1,1/a_0)$, $Q_2=(1/a_0,1/a_1)$ and $Q_3=(1/a_0,1/a_0)$. We emphasize that the region $(Q_1Q_2Q_3)^o$ intentionally contains $(1/p^\ast,1/p^\ast)$.
	
	\vspace{2ex}
	
	We establish the following steps to gain the boundness of $\mathtt{T}_{\sigma,\delta,p^\ast}$. First, the Lions--Peetre formula (see, \cite{Lions})
	\begin{equation*}
		(L_t^2 L_x^{p_0}, L_t^2 L_x^{p_1})_{\theta,2}= L_t^2 L_x^{p_\theta ,2},\qquad\text{whenever }p_0\neq p_1,\quad \frac{1}{p_\theta}=\frac{1-\theta}{p_0}+\frac{\theta}{p_1},\quad 0<\theta<1
	\end{equation*}
	gives $(L_t^2 L_x^{a_0}, L_t^2 L_x^{a_1})_{1/3,2}= L_t^2 L_x^{p^\ast ,2}$, where $L^{p^\ast,2}$ denotes a Lorentz space. Moreover, applying the following interpolation identity (see, \cite{Bergh}, Theorem 5.6.1)
	\begin{equation*}
		(\ell_{\beta_0}^\infty,\ell_{\beta_1}^\infty)_{\theta,1}=\ell_{\beta_\theta}^1,\qquad\text{whenever }\beta_0\neq\beta_1,\quad \beta_\theta=(1-\theta)\beta_0+\theta\beta_1,\quad 0<\theta<1,
	\end{equation*}
	one obtains $(\ell_{\gamma(a_0,a_0)}^\infty,\ell_{\gamma(a_0,a_1)}^\infty)_{2/3,1}=\ell_0^1=\ell^1$.
	
	\vspace{2ex}
	
	According to the interpolation result (\ref{3H}), we obtain a bounded bilinear transformation
	\begin{equation*}
		\mathtt{T}_{\sigma,\delta,p^\ast}: L_t^2 L_x^{p^\ast,2} \times L_t^2 L_x^{p^\ast,2} \longrightarrow \ell^1.
	\end{equation*}We recall that it suffices to show that $\mathtt{T}_{\sigma,\delta,p^\ast}$ is $L_t^2L_x^p\times L_t^2L_x^p\longrightarrow\ell^1$ bounded. Since $p^\ast<2d/(d+1)<2$, which ensures $L_x^{p^\ast}=L_x^{p^\ast,p^\ast} \subset L_x^{p^\ast,2} $, it shows that \begin{equation*}
		\mathtt{T}_{\sigma,\delta,p^\ast}:L_t^2 L_x^p\times L_t^2 L_x^p\longrightarrow \ell^1
	\end{equation*}is bounded whenever $1<p=p^\ast<2d/(d+1)$. In conclusion, once we prove Lemma \ref{33}, we finish the proof of (\ref{3E}) for all $(d+1)/2d<1/p<1$, all $m,n\in\{1,2,3\}$ and all $\sigma\in\{0,1\}$. For more details about Lorentz spaces and their real interpolation theory, we refer the reader to \cite{Bergh}. 
	
	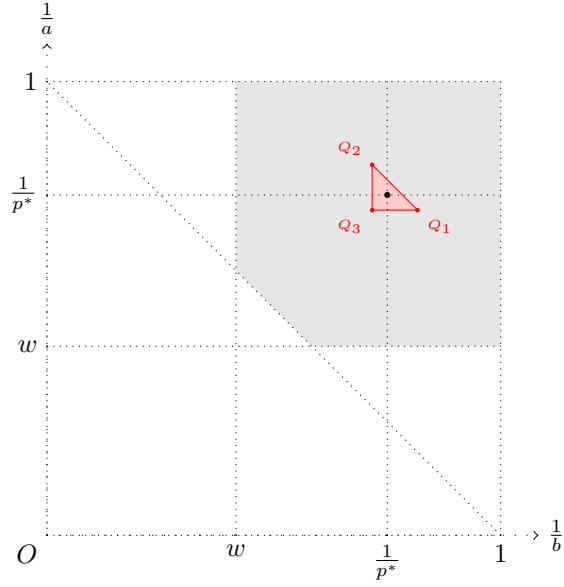
\begin{figure}[H]
		\centering
		\begin{tikzpicture}
			\draw [->, loosely dotted, thin] (0, 0) -- (6.5, 0) node [right] {$\frac{1}{b}$};
			\draw [->, loosely dotted, thin] (0, 0) -- (0, 6.5) node [above] {$\frac{1}{a}$};
			\fill [lightgray!40!white] (2.5, 6) -- (6, 6) -- (6, 2.5) -- (3.5, 2.5) -- (2.5, 3.5) -- cycle;
			\fill [red!20!white] (4.9,4.3) -- (4.3,4.3) -- (4.3,4.9) -- cycle;
			\draw [dotted, thin] (6, 0) -- (6, 6);
			\draw [dotted, thin] (6, 0) -- (0, 6);
			\draw [dotted, thin] (0, 6) -- (6, 6);
			\draw [dotted, thin] (0, 0) -- (0, 6);
			\draw [dotted, thin] (0, 0) -- (6, 0);
			\draw [dotted, thin] (0, 4.5) -- (6, 4.5);
			\draw [dotted, thin] (4.5, 0) -- (4.5, 6);
			\draw [dotted, thin]  (2.5, 0) -- (2.5, 6);
			\draw [dotted, thin]  (0, 2.5) -- (6, 2.5);
			\fill (0, 0) circle (0.01cm) node [below left] {$O$};
			\fill (0, 6) circle (0.01cm) node [left] {$1$};
			\fill (6, 0) circle (0.01cm) node [below] {$1$};
			\fill (0, 2.5) circle (0.01cm) node [left] {$w$};
			\fill (2.5, 0) circle (0.01cm) node [below] {$w$};
			\fill (0, 4.5) circle (0.01cm) node [left] {$\frac{1}{p^\ast}$};
			\fill (4.5, 0) circle (0.01cm) node [below] {$\frac{1}{p^\ast}$};
			\fill (4.5, 4.5) circle (0.04cm);
			\fill [red,thick] (4.9,4.3) circle (0.03cm) node [below right] {\tiny$Q_1$};
			\fill [red,thick] (4.3,4.3) circle (0.03cm) node [below left] {\tiny$Q_3$};
			\fill [red,thick] (4.3,4.9) circle (0.03cm) node [above left] {\tiny$Q_2$};
			\draw [red] (4.9,4.3) -- (4.3,4.3);
			\draw [red] (4.3,4.3) -- (4.3,4.9);
			\draw [red] (4.3,4.9) -- (4.9,4.3);
		\end{tikzpicture}	
		\caption{The interpolation region, with $w=\max\left\{\frac{d+1}{d}-\frac{1}{p^\ast},\frac{1}{p^\ast}-\frac{1}{d}\right\}$.}
	\end{figure}

	\begin{proof}[Proof of Lemma \ref{33}.]
		First, we note that \begin{equation}\label{ggg}
			|\mathcal{T}_{k,0,2s^\ast-d,p^\ast}(F,G)|\lesssim|\mathcal{T}_{k,0,0,p^\ast}(F,G)|
		\end{equation}is true for every non-negative $F$ and $G$, since $d-2s^\ast>0$. For each $\mathcal{T}_{k,\sigma,\delta,p^\ast}$, by changing variables, we have\begin{align*}
			&\quad|\mathcal{T}_{k,\sigma,\delta,p^\ast}(F,G)|\\&\lesssim2^{-k(d-2s^\ast-\delta)}\int_{\mathbb{R}^{2d+2}}\kappa_k(t-t')|F(t,x)||G(t',y)|\nu_\sigma(x,y)|x-y|^{-2\delta}dxdydtdt'\\&=2^{-k(d-2s^\ast)}2^{kd}\int_{\mathbb{R}^{2d+2}}\kappa_k(t-t')|F(t,2^{\frac{k}{2}}x)||G(t',2^{\frac{k}{2}}y)|\nu_\sigma(2^{\frac{k}{2}}x,2^{\frac{k}{2}}y)|x-y|^{-2\delta}dxdydtdt'.
		\end{align*}Since $|t-t'|\sim 2^k$ on the support of $\kappa_k$, we replace $\nu_0(2^{k/2}x,2^{k/2}y)$ and $\nu_1(2^{k/2}x,2^{k/2}y)$, respectively with $1_{|x-y|^2\lesssim1}$ and 
		$1_{|x-y|^2\gtrsim1}$. 
		
		\vspace{2ex}
		
		Consider the case $\sigma=0$. Proceeding with Young's convolution inequality, one can easily obtain the following estimate holds for $\delta=0$, and thus for $\delta=2s^\ast-d$. \begin{equation*}
			|\mathcal{T}_{k,0,\delta,p^\ast}(F,G)|\lesssim2^{-k(d-2s^\ast)}2^{kd}2^{-k\frac{d}{2}\left(\frac{1}{a}+\frac{1}{b}\right)}\int_{\mathbb{R}^{2}}\kappa_k(t-t')\Vert F(t,\cdot)\Vert_{L_x^a} \Vert G(t',\cdot)\Vert_{L_x^b} dtdt'.
		\end{equation*}Regarding $\sigma=1$, by applying Young's convolution inequality, similarly we have \begin{equation*}
			|\mathcal{T}_{k,1,\delta,p^\ast}(F,G)|\lesssim2^{-k(d-2s^\ast)}2^{kd}2^{-k\frac{d}{2}\left(\frac{1}{a}+\frac{1}{b}\right)}\int_{\mathbb{R}^{2}}\kappa_k(t-t')\Vert F(t,\cdot)\Vert_{L_x^a} \Vert G(t',\cdot)\Vert_{L_x^b}\|1_{|\cdot|^2\gtrsim1}|\cdot|^{-2\delta}\|_{L^c} dtdt',
		\end{equation*}where $1/c=2-1/a-1/b$. To see that the convergence condition $2\delta c>d$ holds, we remark that $4s^\ast c>d$ holds if $1/a+1/b>2-2/p^\ast+1/d$; $2(d-2s^\ast)c>d$ holds if $1/a+1/b>2/p^\ast-2/d$; $dc>d$ holds if $1/a+1/b>1$. Thus,\begin{equation}\label{3I}
			|\mathcal{T}_{k,\sigma,\delta,p^\ast}(F,G)|\lesssim2^{-k(d-2s^\ast)}2^{kd}2^{-k\frac{d}{2}\left(\frac{1}{a}+\frac{1}{b}\right)}\int_{\mathbb{R}^{2}}\kappa_k(t-t')\Vert F(t,\cdot)\Vert_{L_x^a} \Vert G(t',\cdot)\Vert_{L_x^b} dtdt'
		\end{equation}holds for every $(\sigma,\delta)\in\{(0,0),(0,2s^\ast-d),(1,2s^\ast),(1,d-2s^\ast),(1,d/2)\}$. 
		
		\vspace{2ex}
		
		Finally, we apply Young's convolution inequality for temporal variables $t,t'$ in (\ref{3I}) and obtain\begin{align*}
			|\mathcal{T}_{k,\sigma,\delta,p^\ast}(F,G)|\lesssim&2^{-k(d-2s^\ast)}2^{kd}2^{-k\frac{d}{2}\left(\frac{1}{a}+\frac{1}{b}\right)}2^{2k}\int_{\mathbb{R}^{2}}\kappa(t-t')\Vert F(2^kt,\cdot)\Vert_{L_x^a} \Vert G(2^kt',\cdot)\Vert_{L_x^b} dtdt'\\\leq&2^{-k(d-2s^\ast)}2^{k\beta(a,b)}\Vert F\Vert_{L_t^2 L_x^a}\Vert G\Vert_{L_t^2 L_x^b}.
		\end{align*}Since $\beta(p^\ast,p^\ast)=d-2s^\ast$, we obtain (\ref{3F}) and conclude Lemma \ref{33}.
	\end{proof}

	\section{The fractional Schrödinger propagator: Proof of Theorem \ref{main2}.}
	
	Let $d\geq2$. Suppose $4d/(d-1)<r<\infty$. Define \begin{equation*}
		s_\mu=\frac{d}{2}-\frac{d}{r}-\frac{\mu}{4},
	\end{equation*}and recall that $\widetilde{r}=2r/(r-2)$. For every $\mu\in(0,\infty)\setminus{1}$, we first prepare a frequency localised setting. Fix a positive, sufficiently small $\varepsilon_0<1/10\mu$ (say). To prove Theorem \ref{main2}, it suffices to show that \begin{equation}\label{41}
		\left\Vert W_1e^{it|\nabla|^\mu}P_\mu^2|\nabla|^{-2s_\mu}e^{-it|\nabla|^\mu}W_2\right\Vert_{\mathfrak{S}^2}\lesssim\Vert W_1\Vert_{L_t^4 L_x^{\widetilde{r}}} \Vert W_2\Vert_{L_t^4 L_x^{\widetilde{r}}}
	\end{equation}holds for all $\mu\in(0,\infty)\setminus{1}$, where $P_\mu$ is given by $\widehat{P_\mu f}(\xi)=\phi_\mu(\xi)\widehat{f}(\xi)$ with \begin{equation*}
		\phi_\mu\in C_c^\infty;\quad\text{supp}\phi_\mu=\left\{\xi;|\xi|\leq\frac{1}{\mu}\right\};\quad\phi_\mu(\xi)\leq1;\quad\phi_\mu(\xi)\equiv1,\text{ if } |\xi|\leq \frac{1}{\mu}-\varepsilon_0.
	\end{equation*}
	
	We denote $p=\widetilde{r}/2$, and remark that $4s_\mu=2d/p-\mu$. To derive (\ref{41}), it suffices to show that \begin{equation}\label{42}
		|T_\mu(F,G)|\lesssim\Vert F\Vert_{L_t^2 L_x^p}\Vert G\Vert_{L_t^2 L_x^p},
	\end{equation}holds for all $\mu\in(0,\infty)\setminus{1}$, and for all non-negative $F,G\in L_t^2 L_x^p$. Here \begin{equation*}
		T_\mu(F,G)=\int_{\mathbb{R}^{2d+2}}F(t,x)G(t',y)\sum_{j_0,j_1} \left|I_{j_0}^\mu(t-t',x-y)I_{j_1}^\mu(t-t',x-y)\right|dxdydtdt',
	\end{equation*}with \begin{equation*}
		I_j^\mu(t,x)=2^{-j(d-2s_\mu)}\int_{\mathbb{R}^d}e^{i(2^{-j}x\cdot\xi-2^{-\mu j}t|\xi|^\mu)}\psi_\mu(\xi)|\xi|^{-2s_\mu}d\xi.
	\end{equation*}The dyadic decomposition function $\psi_\mu$ satisfies\begin{equation*}
		\psi_\mu\in C_c^\infty;\quad\phi_\mu^2(\xi)=\sum_{j=1}^{\infty}\psi_\mu(2^j\xi),
	\end{equation*}but we will set up the support of $\psi_\mu$ later.
	
	\vspace{2ex}
	
	We establish the proof of (\ref{42}) by mimicking the structure in Section 3. Since the derivative of phase function $\Phi(\xi)=2^{-j}x\cdot\xi-2^{-\mu j}t|\xi|^\mu$ is $\nabla\Phi(\xi)=2^{-j}x-\mu2^{-\mu j}t|\xi|^{\mu-2}\xi$, which means that the value of $|\nabla\Phi(\xi)|$ is controlled by $|\xi|^{\mu-1}$, we need to separate this proof into the $0<\mu<1$ case, and the $\mu>1$ case. However, the proofs for both cases entail lengthy computations and present no technical differences. Since the case $\mu>1$ more closely resembles the Schrödinger case $\mu=2$, which we have explicitly discussed in the previous section, here we only provide details for $0<\mu<1$.
	
	\vspace{2ex}
	
	We start from deducing oscillatory estimates. Set\begin{equation*}
		\text{supp}\psi_\mu=\left\{\xi;\mu\leq|\xi|^{1-\mu}\leq2\mu\right\}.
	\end{equation*}Then, the triangle inequality implies \begin{equation}\label{43}
		|\nabla\Phi(\xi)|\gtrsim\begin{cases}
			2^{-j}|x|,& 2|t|\leq2^{(\mu-1)j}|x|;\\2^{-\mu j}|t|,& 2^{(\mu-1)j}|x|\leq2^{-2}|t|.
		\end{cases}
	\end{equation}We give the following notation for collections of $j$ with fixed $x,y,t,t'$. \begin{itemize}
		\item $\mathbf{B}_1=\{j;1/4|t-t'|\leq2^{(\mu-1)j}|x-y|\leq2|t-t'|\}$;
		\item $\mathbf{B}_2=\{j;2|t-t'|\leq2^{(\mu-1)j}|x-y|\}$;
		\item $\mathbf{B}_3=\{j;2^{(\mu-1)j}|x-y|\leq1/4|t-t'|\}$.
	\end{itemize}We recall that for $j\in\mathbf{B}_1$, we have a non-zero Hessian of $\Phi$ on the support of $\psi_\mu$, and\begin{equation*}
		2^{-j}\sim\left(\frac{|t-t'|}{|x-y|}\right)^\frac{1}{1-\mu}.
	\end{equation*}
	
	Building on the argument in Section \ref{s1}, we deduce the following inequalities
	\begin{equation}\label{44}
		\sum_{j\in\mathbf{B}_1}|I_j^\mu(t-t',x-y)|\lesssim2^{-j(d-2s_\mu)}(2^{-\mu j}|t-t'|)^{-d/2}\sim2^{-j(d-2s_\mu)}(2^{-j}|x-y|)^{-d/2};
	\end{equation}\begin{equation}\label{45}
		\sum_{j\in\mathbf{B}_2\cup\mathbf{B}_3}|I_j^\mu(t-t',x-y)|\lesssim\frac{1}{|x-y|^{d-2s_\mu}};
	\end{equation}\begin{equation}\label{46}
		\sum_{j\in\mathbf{B}_2\cup\mathbf{B}_3}|I_j^\mu(t-t',x-y)|\lesssim\frac{1}{|t-t'|^\frac{d-2s_\mu}{\mu}},
	\end{equation}and \begin{equation}\label{47}
		|I_j^\mu(t-t',x-y)|\lesssim 2^{-j(d-2s_\mu)}.
	\end{equation}As with our proof of Theorem \ref{main1} (in particular, see Remark 8), it is instructive to see how estimates the above inequalities combined with Young's convolution inquality just fail to give a proof of Theorem \ref{main2}.
	
	\vspace{2ex}
	
	\noindent\underline{\textbf{Case 1: On $j_0,j_1\in\mathbf{B}_1$}:}
	
	\vspace{2ex}
	
	Firstly, by applying (\ref{44}) and Young's convolution inequality, we deduce \begin{align*}
		&\int_{\mathbb{R}^{2d+2}}|F(t,x)||G(t',y)|1_{|x-y|^\mu\leq|t-t'|}(x,y) \left|I_{j_0}^\mu(t-t',x-y)I_{j_1}^\mu(t-t',x-y)\right|dxdydtdt'\\\lesssim&\int_{|x-y|^\mu\leq|t-t'|}|F(t,x)||G(t',y)|\left(\frac{|t-t'|}{|x-y|}\right)^{\frac{2d-\mu d-4s_\mu}{1-\mu}}\frac{1}{|t-t'|^d}dxdydtdt'\\\leq&\int_{\mathbb{R}^2}\Vert F(t,\cdot)\Vert_{L_x^p}\Vert G(t',\cdot)\Vert_{L_x^p}\left\Vert\frac{1_{|\cdot|^\mu\leq1}}{|\cdot|^{\frac{2d-\mu d-4s_\mu}{1-\mu}}}\right\Vert_{L_x^c}\\&\qquad\qquad\qquad\frac{1}{|t-t'|^{d-\frac{2d-\mu d-4s_\mu}{1-\mu}}}\frac{1}{|t-t'|^{\frac{2d-\mu d-4s_\mu}{\mu(1-\mu)}}}|t-t'|^\frac{2d}{\mu}\frac{1}{|t-t'|^{\frac{2d}{p\mu}}}dtdt',
	\end{align*}where $1/c=2-2/p$. Since $4s_\mu=2d/p-\mu$, we observe the following:
	\begin{itemize}
		\item To ensure convergence of the $L_x^c$ norm, we need\begin{equation}\label{48}
			\frac{2d-\mu d-4s_\mu}{1-\mu}<\frac{d}{c},
		\end{equation}and we note that this is equivalent to $1/p>(d+1)/2d$.
		\item The order of $1/|t-t'|$ is \begin{equation*}
			d-\frac{2d-\mu d-4s_\mu}{1-\mu}+\frac{2d-\mu d-4s_\mu}{\mu(1-\mu)}-\frac{2d}{\mu}+\frac{2d}{p\mu}=d-\frac{2d}{\mu}+\frac{2d}{p\mu}+\frac{1}{\mu}\left(2d-\mu d-\frac{2d}{p}+\mu\right),
		\end{equation*}which equals to $1$.
	\end{itemize}
	
	On the other hand, by applying (\ref{47}), the fact $\#\mathbf{B}_1\sim1$, and Young's convolution inequality, we deduce\begin{align*}
		&\int_{\mathbb{R}^{2d+2}}|F(t,x)||G(t',y)|1_{|x-y|^\mu\geq|t-t'|}(x,y) \left|I_{j_0}^\mu(t-t',x-y)I_{j_1}^\mu(t-t',x-y)\right|dxdydtdt'\\\lesssim&\int_{|x-y|^\mu\geq|t-t'|}|F(t,x)||G(t',y)|2^{-(j_0+j_1)(d-2s_\mu)}dxdydtdt'\\\sim&\int_{|x-y|^\mu\geq|t-t'|}|F(t,x)||G(t',y)|\left(\frac{|t-t'|}{|x-y|}\right)^{\frac{2d-4s_\mu}{1-\mu}}dxdydtdt'\\\leq&\int_{\mathbb{R}^2}\Vert F(t,\cdot)\Vert_{L_x^p}\Vert G(t',\cdot)\Vert_{L_x^p}\left\Vert\frac{1_{|\cdot|^\mu\geq1}}{|\cdot|^{\frac{2d-4s_\mu}{1-\mu}}}\right\Vert_{L_x^c}|t-t'|^{\frac{2d-4s_\mu}{1-\mu}}\frac{1}{|t-t'|^{\frac{2d-4s_\mu}{\mu(1-\mu)}}}|t-t'|^\frac{2d}{\mu}\frac{1}{|t-t'|^{\frac{2d}{p\mu}}}dtdt'.
	\end{align*}
	
	\begin{itemize}
		\item To ensure convergence of the $L_x^c$ norm, we need\begin{equation*}
			\frac{2d-4s_\mu}{1-\mu}>\frac{d}{c},
		\end{equation*}and this holds if and only if $1/p<(2d+1)/2d$, which is trivial.
		\item The order of $1/|t-t'|$ is \begin{equation*}
			-\frac{2d-4s_\mu}{1-\mu}+\frac{2d-4s_\mu}{\mu(1-\mu)}-\frac{2d}{\mu}+\frac{2d}{p\mu}=\frac{1}{\mu}\left(2d-\frac{2d}{p}+\mu\right)-\frac{2d}{\mu}+\frac{2d}{p\mu},
		\end{equation*}which equals to $1$.
	\end{itemize}
	
	\vspace{2ex}
	
	\noindent\underline{\textbf{Case 2: On $j_0,j_1\in\mathbf{B_2\cup B_3}$}:}
	
	\vspace{2ex}
	
	Applying (\ref{45}), (\ref{46}) and the Young's convolution inequality yields \begin{align*}
		&\int_{|x-y|^\mu\geq|t-t'|}|F(t,x)||G(t',y)|\sum_{j_0,j_1\in\mathbf{B_2\cup B_3}} \left|I_{j_0}^\mu(t-t',x-y)I_{j_1}^\mu(t-t',x-y)\right|dxdydtdt'\notag\\\lesssim&\int_{|x-y|^\mu\geq|t-t'|}|F(t,x)||G(t',y)|\frac{1}{|x-y|^{2d-4s_\mu}}dxdydtdt'\notag\\\leq&\int_{\mathbb{R}^2}\Vert F(t,\cdot)\Vert_{L_x^p}\Vert G(t',\cdot)\Vert_{L_x^p}\left\Vert\frac{1_{|\cdot|^\mu\geq1}}{|\cdot|^{2d-4s_\mu}}\right\Vert_{L_x^c}\frac{1}{|t-t'|^{\frac{2d-4s_\mu}{\mu}}}|t-t'|^\frac{2d}{\mu}\frac{1}{|t-t'|^{\frac{2d}{p\mu}}}dtdt',
	\end{align*}and\begin{align*}
		&\int_{|x-y|^\mu\leq|t-t'|}|F(t,x)||G(t',y)|\sum_{j_0,j_1\in\mathbf{B_2\cup B_3}} \left|I_{j_0}^\mu(t-t',x-y)I_{j_1}^\mu(t-t',x-y)\right|dxdydtdt'\notag\\\lesssim&\int_{|x-y|^\mu\leq|t-t'|}|F(t,x)||G(t',y)|\frac{1}{|t-t'|^{\frac{2d-4s_\mu}{\mu}}}dxdydtdt'\\\leq&\int_{\mathbb{R}^2}\Vert F(t,\cdot)\Vert_{L_x^p}\Vert G(t',\cdot)\Vert_{L_x^p}\left\Vert1_{|\cdot|^\mu\leq1}\right\Vert_{L_x^c}\frac{1}{|t-t'|^{\frac{2d-4s_\mu}{\mu}}}|t-t'|^\frac{2d}{\mu}\frac{1}{|t-t'|^{\frac{2d}{p\mu}}}dtdt'.
	\end{align*}
	
	\begin{itemize}
		\item To ensure convergence of the $L_x^c$ norm, we need\begin{equation*}
			2d-4s_\mu>\frac{d}{c},
		\end{equation*}and this is equivalent to $4s_\mu<2d/p$, which is trivial.
		\item The order of $1/|t-t'|$ is \begin{equation*}
			\frac{2d-4s_\mu}{\mu}-\frac{2d}{\mu}+\frac{2d}{p\mu}=\frac{2d}{\mu}-\frac{2d}{p\mu}+1-\frac{2d}{\mu}+\frac{2d}{p\mu}=1.
		\end{equation*}
	\end{itemize}
	
	\noindent\underline{\textbf{Case 3: On $j_0\in\mathbf{B}_1,j_1\in\mathbf{B_2\cup B_3}$}:}
	
	\vspace{2ex}
	
	Applying (\ref{44}), (\ref{45}), (\ref{46}) and the Young's convolution inequality yields 
	\begin{align*}
		&\int_{|x-y|^\mu\leq|t-t'|}|F(t,x)||G(t',y)||I_{j_0}^\mu(t-t',x-y)|\sum_{j_1\in\mathbf{B_2\cup B_3}} \left|I_{j_1}^\mu(t-t',x-y)\right|dxdydtdt'\\\lesssim&\int_{|x-y|^\mu\leq|t-t'|}|F(t,x)||G(t',y)|\left(\frac{|t-t'|}{|x-y|}\right)^{\frac{2d-\mu d-4s_\mu}{2-2\mu}}\frac{1}{|t-t'|^\frac{d}{2}}\frac{1}{|t-t'|^\frac{d-2s_\mu}{\mu}}dxdydtdt'\\\leq&\int_{\mathbb{R}^2}\Vert F(t,\cdot)\Vert_{L_x^p}\Vert G(t',\cdot)\Vert_{L_x^p}\left\Vert\frac{1_{|\cdot|^\mu\leq1}}{|\cdot|^{\frac{2d-\mu d-4s_\mu}{2-2\mu}}}\right\Vert_{L_x^c}\\&\qquad\qquad|t-t'|^{\frac{2d-\mu d-4s_\mu}{2-2\mu}}\frac{1}{|t-t'|^{\frac{2d-\mu d-4s_\mu}{\mu(2-2\mu)}}}\frac{1}{|t-t'|^\frac{d}{2}}\frac{1}{|t-t'|^\frac{d-2s_\mu}{\mu}}|t-t'|^\frac{2d}{\mu}\frac{1}{|t-t'|^{\frac{2d}{p\mu}}}dtdt'.
	\end{align*}
	\begin{itemize}
		\item The convergence of the $L_x^c$ norm follows by (\ref{48}).
		\item The order of $1/|t-t'|$ is \begin{align*}
			&\quad-\frac{2d-\mu d-4s_\mu}{2-2\mu}+\frac{d}{2}+\frac{d-2s_\mu}{\mu}+\frac{2d-\mu d-4s_\mu}{\mu(2-2\mu)}-\frac{2d}{\mu}+\frac{2d}{p\mu}\\&=\frac{2d-\mu d-4s_\mu}{2\mu}+\frac{\mu d}{2\mu}+\frac{2d-4s_\mu}{2\mu}-\frac{2d}{\mu}+\frac{2d}{p\mu}\\&=-\frac{4s_\mu}{\mu}+\frac{2d}{p\mu}\\&=1.
		\end{align*}
	\end{itemize}
	
	On the other hand, applying (\ref{45}), (\ref{46}), (\ref{47}) and Young's convolution inequality yields 
	
	\begin{align*}
		&\int_{|x-y|^\mu\geq|t-t'|}|F(t,x)||G(t',y)||I_{j_0}^\mu(t-t',x-y)|\sum_{j_1\in\mathbf{B_2\cup B_3}} \left|I_{j_1}^\mu(t-t',x-y)\right|dxdydtdt'\\\lesssim&\int_{|x-y|^\mu\geq|t-t'|}|F(t,x)||G(t',y)|\left(\frac{|t-t'|}{|x-y|}\right)^{\frac{d-2s_\mu}{1-\mu}}\frac{1}{|x-y|^{d-2s_\mu}}dxdydtdt'\\\leq&\int_{\mathbb{R}^2}\Vert F(t,\cdot)\Vert_{L_x^p}\Vert G(t',\cdot)\Vert_{L_x^p}\left\Vert\frac{1_{|\cdot|^\mu\geq1}}{|\cdot|^{\frac{d-2s_\mu}{1-\mu}+d-2s_\mu}}\right\Vert_{L_x^c}\\&\qquad\qquad\qquad\qquad\qquad\qquad|t-t'|^{\frac{d-2s_\mu}{1-\mu}}\frac{1}{|t-t'|^{\frac{d-2s_\mu}{\mu(1-\mu)}+\frac{d-2s_\mu}{\mu}}}|t-t'|^\frac{2d}{\mu}\frac{1}{|t-t'|^{\frac{2d}{p\mu}}}dtdt'.
	\end{align*}
	\begin{itemize}
		\item The convergence of the $L_x^c$ norm follows by \begin{equation*}
			\frac{d-2s_\mu}{1-\mu}+d-2s_\mu>2d-\frac{2d}{p}.
		\end{equation*}This is equivalent to $d/p<1-\mu/2+d$, and this holds since $p>1$.
		\item The order of $1/|t-t'|$ is \begin{equation*}
			-\frac{d-2s_\mu}{1-\mu}+\frac{d-2s_\mu}{\mu(1-\mu)}+\frac{d-2s_\mu}{\mu}-\frac{2d}{\mu}+\frac{2d}{p\mu}=-\frac{4s_\mu}{\mu}+\frac{2d}{p\mu}=1.
		\end{equation*}
	\end{itemize}
	
	\vspace{2ex}
	
	Following the structure of the proof of Theorem \ref{main1}, to handle the singularity arising from $1/|t-t'|$, we perform a dyadic decomposition in the temporal variable and use Proposition \ref{Bergh-Löfström} to sum up the pieces. For this, take $1<p^\ast<2d/(d+1)$ and fix it. Let $s_\mu^\ast$ be given by $2s_\mu^\ast=d/p^\ast-\mu/2$. Choose $\kappa\in C_c^\infty(\mathbb{R})$ and $\kappa_k$ as in Section 3.2. Based on the calculations in the preceding three cases, we define bilinear operators $\mathcal{T}_{w,k}$ for $w\in\{1,2,3,4,5,6\}$ as follows:\begin{equation*}
		\mathcal{T}_{1,k}(F,G)=\int_{|x-y|^\mu\leq|t-t'|}\kappa_k(t-t')F(t,x)G(t',y)\left(\frac{|t-t'|}{|x-y|}\right)^{\frac{2d-\mu d-4s_\mu^\ast}{1-\mu}}\frac{1}{|t-t'|^d}dxdydtdt';
	\end{equation*}\begin{equation*}
		\mathcal{T}_{2,k}(F,G)=\int_{|x-y|^\mu\geq|t-t'|}\kappa_k(t-t')F(t,x)G(t',y)\left(\frac{|t-t'|}{|x-y|}\right)^{\frac{2d-4s_\mu^\ast}{1-\mu}}dxdydtdt';
	\end{equation*}\begin{equation*}
		\mathcal{T}_{3,k}(F,G)=\int_{|x-y|^\mu\geq|t-t'|}\kappa_k(t-t')F(t,x)G(t',y)\frac{1}{|x-y|^{2d-4s_\mu^\ast}}dxdydtdt';
	\end{equation*}\begin{equation*}
		\mathcal{T}_{4,k}(F,G)=\int_{|x-y|^\mu\leq|t-t'|}\kappa_k(t-t')F(t,x)G(t',y)\frac{1}{|t-t'|^{\frac{2d-4s_\mu^\ast}{\mu}}}dxdydtdt';
	\end{equation*}\begin{equation*}
		\mathcal{T}_{5,k}(F,G)=\int_{|x-y|^\mu\leq|t-t'|}\kappa_k(t-t')F(t,x)G(t',y)\left(\frac{|t-t'|}{|x-y|}\right)^{\frac{2d-\mu d-4s_\mu^\ast}{2-2\mu}}\frac{1}{|t-t'|^\frac{d}{2}}\frac{1}{|t-t'|^\frac{d-2s_\mu^\ast}{\mu}}dxdydtdt';
	\end{equation*}\begin{equation*}
		\mathcal{T}_{6,k}(F,G)=\int_{|x-y|^\mu\geq|t-t'|}\kappa_k(t-t')F(t,x)G(t',y)\left(\frac{|t-t'|}{|x-y|}\right)^{\frac{d-2s_\mu^\ast}{1-\mu}}\frac{1}{|x-y|^{d-2s_\mu^\ast}}dxdydtdt'.
	\end{equation*}The following statement is an extension of Lemma \ref{33}.
	\begin{lem}\label{4.}
		Let $d\geq2$. Suppose $1<p^\ast<2d/(d+1)$. Let $a,b$ be such that $1/a+1/b>1$. Define \begin{equation*}
			\beta(a,b)=d+1-\frac{d}{\mu}\left(\frac{1}{a}+\frac{1}{b}\right),
		\end{equation*}and denote $\gamma(a,b)=\beta(p^\ast,p^\ast)-\beta(a,b)$. Then, there exists a region $X\subset(0,1)\times(0,1)$, such that the inequality \begin{equation}\label{49}
			|\mathcal{T}_{w,k}(F,G)|\lesssim2^{-k\gamma(a,b)}\Vert F\Vert_{L_t^2 L_x^a}\Vert G\Vert_{L_t^2 L_x^b}
		\end{equation}holds for all $w\in\{1,2,3,4,5,6\}$, $F\in L_t^2 L_x^a$, $G\in L_t^2 L_x^b$, and for all $(1/a,1/b)\in X$. Moreover, $(1/p^\ast,1/p^\ast)\in X^o$.
	\end{lem}
	
		\begin{proof}
		To establish (\ref{49}), we give computational details for each case $\mathcal{T}_{w,k}$. First note that for $\mathcal{T}_{1,k}$, similar with the calculation in Case 1, we deduce \begin{align*}
			&\quad\int_{|x-y|^\mu\leq|t-t'|}\kappa_k(t-t')|F(t,x)||G(t',y)|\frac{1}{|x-y|^{\frac{2d-\mu d-4s_\mu^\ast}{1-\mu}}}\frac{1}{|t-t'|^{d-\frac{2d-\mu d-4s_\mu^\ast}{1-\mu}}}dxdydtdt'\\&\lesssim2^{-k\left(d-\frac{2d-\mu d-4s_\mu^\ast}{1-\mu}\right)}\int_{|x-y|^\mu\leq2^k}\kappa_k(t-t')|F(t,x)||G(t',y)|\frac{1}{|x-y|^{\frac{2d-\mu d-4s_\mu^\ast}{1-\mu}}}dxdydtdt'\\&\leq2^{-k\left(d-\frac{2d-\mu d-4s_\mu^\ast}{1-\mu}\right)}2^{-\frac{k}{\mu}\left(\frac{2d-\mu d-4s_\mu^\ast}{1-\mu}\right)}2^{2d\frac{k}{\mu}}2^{-k\frac{d}{\mu}\left(\frac{1}{a}+\frac{1}{b}\right)}2^k\Vert F\Vert_{L_t^2 L_x^a}\Vert G\Vert_{L_t^2 L_x^b}\left\Vert\frac{1_{|\cdot|^\mu\leq1}}{|\cdot|^{\frac{2d-\mu d-4s_\mu^\ast}{1-\mu}}}\right\Vert_{L_x^c}\\&=2^{-k\gamma(a,b)}\Vert F\Vert_{L_t^2 L_x^a}\Vert G\Vert_{L_t^2 L_x^b}\left\Vert\frac{1_{|\cdot|^\mu\leq1}}{|\cdot|^{\frac{2d-\mu d-4s_\mu^\ast}{1-\mu}}}\right\Vert_{L_x^c}.
		\end{align*}Similar calculations in Case 1, Case 2, and Case 3 give \begin{equation*}
		\mathcal{T}_{w,k}(F,G)\lesssim2^{-k\gamma(a,b)}K_w\Vert F\Vert_{L_t^2 L_x^a}\Vert G\Vert_{L_t^2 L_x^b}
		\end{equation*}for every $w\in\{1,2,3,4,5,6\}$, where $K_w$ denotes the value of $L_x^c$ term for each $w$, as in Case 1, Case 2, and Case 3, namely, \begin{align*}
		&K_1=\left\Vert\frac{1_{|\cdot|^\mu\leq1}}{|\cdot|^{\frac{2d-\mu d-4s_\mu^\ast}{1-\mu}}}\right\Vert_{L_x^c};\qquad K_2=\left\Vert\frac{1_{|\cdot|^\mu\geq1}}{|\cdot|^{\frac{2d-4s_\mu^\ast}{1-\mu}}}\right\Vert_{L_x^c};\qquad K_3=\left\Vert\frac{1_{|\cdot|^\mu\geq1}}{|\cdot|^{2d-4s_\mu^\ast}}\right\Vert_{L_x^c};\\&K_4=\left\Vert1_{|\cdot|^\mu\leq1}\right\Vert_{L_x^c};\qquad K_5=\left\Vert\frac{1_{|\cdot|^\mu\leq1}}{|\cdot|^{\frac{2d-\mu d-4s_\mu^\ast}{2-2\mu}}}\right\Vert_{L_x^c};\qquad K_6=\left\Vert\frac{1_{|\cdot|^\mu\geq1}}{|\cdot|^{\frac{d-2s_\mu^\ast}{1-\mu}+d-2s_\mu^\ast}}\right\Vert_{L_x^c}.
		\end{align*}
		
		\vspace{2ex}
		
		Note that our purpose is to figure out some $X\subset(0,1)\times(0,1)$ such that $(1/p^\ast,1/p^\ast)\in X^o$, satisfies $K_w<\infty$ for every $w$ and for $(1/b,1/a)\in X$. To see that, we consider a special case \begin{equation*}
			\frac{1}{a}=\frac{1}{b}=\frac{1}{p^\ast}+\eta.
		\end{equation*}Under this assumption, we have \begin{equation*}
			\frac{1}{c}=2-\frac{2}{p^\ast}-2\eta.
		\end{equation*}To state the following precisely, we denote our target point $p_0=(1/p^\ast,1/p^\ast)$.\begin{itemize}
			\item For $K_1<\infty$, we need\begin{equation*}
				\frac{2d-\mu d-4s_\mu^\ast}{1-\mu}<d\left(2-\frac{2}{p^\ast}-2\eta\right),
			\end{equation*}which means \begin{equation*}
			\eta<\frac{\mu}{2-2\mu}\left(\frac{2}{p^\ast}-\frac{d+1}{d}\right),
			\end{equation*}shows that we can choose \begin{equation*}
				\eta_1=\frac{\mu}{4-4\mu}\left(\frac{2}{p^\ast}-\frac{d+1}{d}\right)
			\end{equation*}to ensure that\begin{equation*}
				p_0\in X_1=\left(0,\frac{1}{p^\ast}+\eta_1\right)\times\left(0,\frac{1}{p^\ast}+\eta_1\right).
			\end{equation*}
			\item For $K_2$, if we restrict \begin{equation*}
				\eta>-\frac{\mu}{2d(1-\mu)},
			\end{equation*}then we have \begin{equation*}
				\frac{2d-4s_\mu^\ast}{1-\mu}>d\left(2-\frac{2}{p^\ast}-2\eta\right),
			\end{equation*}which means $K_2<\infty$, and we can choose \begin{equation*}
				\eta_2=-\frac{\mu}{4d(1-\mu)}
			\end{equation*}to ensure that\begin{equation*}
				p_0\in X_2=\left(\frac{1}{p^\ast}+\eta_2,1\right)\times\left(\frac{1}{p^\ast}+\eta_2,1\right).
			\end{equation*}
			\item For $K_3<\infty$, easily we deduce\begin{equation*}
				\mu>-2d\eta.
			\end{equation*}Thus, we can take \begin{equation*}
				\eta_3=-\frac{\mu}{4d}
			\end{equation*}to ensure that\begin{equation*}
				p_0\in X_3=\left(\frac{1}{p^\ast}+\eta_3,1\right)\times\left(\frac{1}{p^\ast}+\eta_3,1\right).
			\end{equation*}
			\item For $K_4$, \begin{equation*}
				X_4=\left\{\left(\frac{1}{a},\frac{1}{b}\right);1<a,b<2,\frac{1}{a}+\frac{1}{b}>1\right\}
			\end{equation*}trivially ensures $K_4<\infty$ and $p_0\in X_4$.
			\item For $K_5<\infty$,\begin{equation*}
				2d\eta<2d-\frac{2d}{p^\ast}-\frac{2d-\mu d-4s_\mu^\ast}{2-2\mu}
			\end{equation*}means $\eta$ can be chosen as a positive number. More precisely, if $0<\mu<1/2$, then \begin{equation*}
				(2-2\mu)\left(2d-\frac{2d}{p^\ast}-\frac{2d-\mu d-4s_\mu^\ast}{2-2\mu}\right)>\mu(d-1)>0;
			\end{equation*}if $1/2\leq\mu<1$, then\begin{equation*}
				2d-\frac{2d}{p^\ast}-\frac{2d-\mu d-4s_\mu^\ast}{2-2\mu}>\frac{d-1}{2}>0.
			\end{equation*}We then choose \begin{equation*}
				\eta_5=\frac{1}{4d}\left(2d-\frac{2d}{p^\ast}-\frac{2d-\mu d-4s_\mu^\ast}{2-2\mu}\right)
			\end{equation*}to ensure that\begin{equation*}
				p_0\in X_5=\left(0,\frac{1}{p^\ast}+\eta_5\right)\times\left(0,\frac{1}{p^\ast}+\eta_5\right).
			\end{equation*}
			\item For $K_6<\infty$, we need \begin{equation*}
				\frac{d-2s_\mu^\ast}{1-\mu}+d-2s_\mu^\ast>d\left(2-\frac{2}{p^\ast}-2\eta\right),
			\end{equation*}so\begin{equation*}
			\eta>\frac{\mu}{4d(1-\mu)}\left(\frac{2d}{p^\ast}+\mu-2d-2\right).
			\end{equation*}Since $p^\ast>1$ and $0<\mu<1$, we can choose \begin{equation*}
				\eta_6=\frac{\mu}{8d(1-\mu)}\left(\frac{2d}{p^\ast}+\mu-2d-2\right)
			\end{equation*}to ensure that\begin{equation*}
				p_0\in X_6=\left(\frac{1}{p^\ast}+\eta_6,1\right)\times\left(\frac{1}{p^\ast}+\eta_6,1\right).
			\end{equation*}
		\end{itemize}In conclusion, if we take \begin{align*}
			X=&\bigcap_{w=1}^6 X_w\\=&\left\{\left(\frac{1}{p^\ast}+\max_{i\in\{2,3,6\}}\{\eta_i\},\frac{1}{p^\ast}+\min_{i\in\{1,5\}}\{\eta_i\}\right)\times\left(\frac{1}{p^\ast}+\max_{i\in\{2,3,6\}}\{\eta_i\},\frac{1}{p^\ast}+\min_{i\in\{1,5\}}\{\eta_i\}\right)\right\}\bigcap X_4,
		\end{align*}then $(1/p^\ast,1/p^\ast)\in(X)^o$. Also, $K_w<\infty$ for every $w\in\{1,2,3,4,5,6\}$ and for every $(1/b,1/a)\in X$, which completes our proof of Lemma \ref{4.}.
	\end{proof}
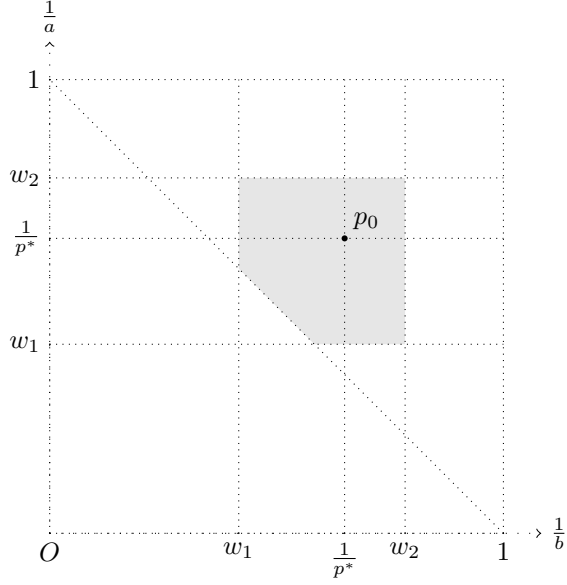
\begin{figure}[H]
	\centering
	\begin{tikzpicture}
		\draw [->, loosely dotted, thin] (0, 0) -- (6.5, 0) node [right] {$\frac{1}{b}$};
		\draw [->, loosely dotted, thin] (0, 0) -- (0, 6.5) node [above] {$\frac{1}{a}$};
		\fill [gray!20!white] (2.5, 4.7) -- (4.7, 4.7) -- (4.7, 2.5) -- (3.5, 2.5) -- (2.5, 3.5) -- cycle;
		\draw [dotted, thin] (6, 0) -- (6, 6);
		\draw [dotted, thin] (6, 0) -- (0, 6);
		\draw [dotted, thin] (0, 6) -- (6, 6);
		\draw [dotted, thin] (0, 0) -- (0, 6);
		\draw [dotted, thin] (0, 0) -- (6, 0);
		\draw [dotted, thin] (0, 3.9) -- (6, 3.9);
		\draw [dotted, thin] (3.9, 0) -- (3.9, 6);
		\draw [dotted, thin]  (2.5, 0) -- (2.5, 6);
		\draw [dotted, thin]  (0, 2.5) -- (6, 2.5);
		\draw [dotted, thin]  (4.7, 0) -- (4.7, 6);
		\draw [dotted, thin]  (0, 4.7) -- (6, 4.7);
		\fill (0, 0) circle (0.01cm) node [below] {$O$};
		\fill (0, 6) circle (0.01cm) node [left] {$1$};
		\fill (6, 0) circle (0.01cm) node [below] {$1$};
		\fill (0, 2.5) circle (0.01cm) node [left] {$w_1$};
		\fill (2.5, 0) circle (0.01cm) node [below] {$w_1$};
		\fill (0, 4.7) circle (0.01cm) node [left] {$w_2$};
		\fill (4.7, 0) circle (0.01cm) node [below] {$w_2$};
		\fill (0, 3.9) circle (0.01cm) node [left] {$\frac{1}{p^\ast}$};
		\fill (3.9, 0) circle (0.01cm) node [below] {$\frac{1}{p^\ast}$};
		\fill (3.9, 3.9) circle (0.04cm) node [above right] {$p_0$};
		\fill (8.3, 6) circle (0.01cm) node [below] {Here,};
		\fill (10, 4) circle (0.01cm) node [below] {$w_2:=\frac{1}{p^\ast}+\min_{i\in\{1,5\}}\{\eta_i\}$.};
		\fill (10, 5) circle (0.01cm) node [below] {$w_1:=\frac{1}{p^\ast}+\max_{i\in\{2,3,6\}}\{\eta_i\}$;};
	\end{tikzpicture}	
	\caption{The interpolation region $X$.}
	\end{figure}
	
	The proof of (\ref{41}) for $0<\mu<1$ then follows by the argument in Section \ref{s2}, by applying (\ref{49}) and Proposition \ref{Bergh-Löfström}.


\begin{thebibliography}{1}
		
		\bibitem{Bergh}
		J.~Bergh, J.~Löfström,
		\newblock {\em Interpolation Spaces: An Introduction},
		\newblock Springer-Verlag, New York, 1976.
		
		\bibitem{Bez}
		N.~Bez, Y.~Hong, S.~Lee, S.~Nakamura, Y.~Sawano,
		\newblock {\em On the Strichartz estimates for orthonormal systems of initial data with regularity},
		\newblock Adv. Math. 354 (2019), 106736.
		
		\bibitem{Bez2}
		N.~Bez, S.~Kinoshita, S.~Shiraki,
		\newblock {\em Boundary Strichartz estimates and pointwise convergence for orthonormal systems},
		\newblock Trans. London Math. Soc. 11 (2024), no. 1, e70002.
		
		\bibitem{Bez3}
		N.~Bez, S.~Lee, S.~Nakamura,
		\newblock {\em Strichartz estimates for orthonormal families of initial data and weighted oscillatory integral estimates},
		\newblock Forum Math. Sigma 9 (2021), e1, 1--52.
		
		\bibitem{Chen}
		T.~Chen, Y.~Hong, N.~Pavlović,
		\newblock {\em Global well-posedness of the NLS system for infinitely many fermions},
		\newblock Arch. Ration. Mech. Anal. 224 (2017), 91--123.
		
		\bibitem{Cwikel}
		M.~Cwikel,
		\newblock {\em On $(L^{p_0}(A_0),L^{p_1}(A_1))_{\theta,q}$},
		\newblock Proc. Amer. Math. Soc. 44 (1974), 286--292.
		
		\bibitem{Feng}
		G.~Feng, M.~Song, H.~Wu,
		\newblock {\em Strichartz estimates involving orthonormal systems at the critical summability exponent},
		\newblock arXiv:2507.14974.
		
		\bibitem{Frank1}
		R.~Frank, M.~Lewin, E.~Lieb, R.~Seiringer,
		\newblock {\em Strichartz inequality for orthonormal functions},
		\newblock J. Eur. Math. Soc. 16 (2014), 1507--1526.
		
		\bibitem{Frank2}
		R.~Frank, J.~Sabin,
		\newblock {\em The Stein-Tomas inequality in trace ideals},
		\newblock Séminaire Laurent Schwartz - Équations aux dérivées partielles et applications. Année 2015--2016, Exp. No. XV, 12 pp., Ed. Éc. Polytech., Palaiseau, 2017.
		
		\bibitem{Frank3}
		R.~Frank, J.~Sabin,
		\newblock {\em Restriction theorems for orthonormal functions, Strichartz inequalities, and uniform Sobolev estimates},
		\newblock Amer. J. Math. 139 (2017), 1649--1691.
		
		\bibitem{Frank4}
		R.~Frank,
		\newblock {\em Lieb-Thirring inequalities and other functional inequalities for orthonormal systems},
		\newblock Proc. Int. Cong. Math. 2022, Vol. 5, EMS Press, Berlin, 2023, 3756--3774.
		
		\bibitem{Ginibre-Velo}
		J.~Ginibre, G.~Velo,
		\newblock {\em Smoothing properties and retarded estimates for some dispersive evolution equations},
		\newblock Commun. Math. Phys. 144 (1992), 163--188.
		
		\bibitem{Z.Guo}
		Z.~Guo, J.~Li, K.~Nakanishi, L.~Yan,
		\newblock {\em On the boundary Strichartz estimates for wave and Schrödinger equations},
		\newblock J. Differential Equations 265 (2018), 5656--5675.
		
		\bibitem{Keel-Tao}
		M.~Keel, T.~Tao,
		\newblock {\em Endpoint Strichartz estimates},
		\newblock Amer. J. Math. 120 (1998), 955--980.
		
		\bibitem{Laskin1}
		N.~Laskin,
		\newblock {\em Fractional quantum mechanics and Lévy path integrals},
		\newblock Phys. Lett. A 268 (2000), 298--305.
		
		\bibitem{Laskin2}
		N.~Laskin,
		\newblock {\em Fractional Schrödinger equation},
		\newblock Phys. Rev. E 66 (2002), 056108.
		
		\bibitem{Lions}
		J.~Lions, J.~Peetre,
		\newblock {\em Sur une classe d’espaces d’interpolation},
		\newblock Inst. Hautes Etudes Sci. Publ. Math. 19 (1964), 5--68.
		
		\bibitem{Montgomery-Smith}
		S.~J.~Montgomery-Smith,
		\newblock {\em Time decay for the bounded mean oscillation of solutions of the Schrödinger and wave equations},
		\newblock Duke Math. J. 91 (1998), 393--408.
		
		\bibitem{C.Muscalu}
		C.~Muscalu, W.~Schlag,
		\newblock {\em Classical and Multilinear Harmonic Analysis: Volume I},
		\newblock Cambridge Univ. Press, 2013.
		
		\bibitem{Simon}
		B.~Simon,
		\newblock {\em Trace ideals and their applications},
		\newblock London Mathematical Society Lecture Note Series 35, Cambridge Univ. Press, 1979.
		
		\bibitem{M.Stein 1}
		E.~M.~Stein, R.~Shakarchi,
		\newblock {\em Functional Analysis},
		\newblock Princeton Univ. Press, 2011.
		
		\bibitem{Strichartz}
		R.~S.~Strichartz,
		\newblock {\em Restrictions of Fourier transforms to quadratic surfaces and decay of solutions of wave equations},
		\newblock Duke Math. J. 44 (1977), 705--714.
		
	\end{thebibliography}
	\end{document}